\documentclass[10pt]{article}
\usepackage[a4paper, total={6in, 8in}]{geometry}     

\usepackage{graphicx}
\usepackage{multirow}
\usepackage{booktabs}
\usepackage{csvsimple}
\usepackage{rotating}
\usepackage{url}
\usepackage{enumerate}

\usepackage{amsmath}
\usepackage{amssymb}
\usepackage{amsthm}

\usepackage{authblk}
\usepackage{setspace}
\usepackage[title]{appendix}

\usepackage{tikz}
\usepackage{pgfplots}
\pgfplotsset{set layers}
\pgfplotsset{compat=1.14}
\usetikzlibrary{arrows}

\usepackage{natbib}

\newtheorem{definition}{Definition}
\newtheorem{lemma}{Lemma}
\newtheorem{theorem}{Theorem}
\newtheorem{corollary}{Corollary}

\newtheorem{example}{Example}

\title{On the stochastic inventory problem under order capacity constraints}
\author[1]{Roberto Rossi\thanks{Corresponding author; 29 Buccleuch place, EH8 9JS, Edinburgh, UK; email: roberto.rossi@ed.ac.uk}}
\author[2]{Zhen Chen}
\author[3]{S. Armagan Tarim}
\affil[1]{Business School, University of Edinburgh, UK}
\affil[2]{College of Economics and Management, Southwest University, China}
\affil[3]{Business School, University College Cork, Ireland}
\date{}                                           

\begin{document}
\maketitle

\onehalfspacing

\begin{abstract}
We consider the single-item single-stocking location stochastic inventory system under a fixed ordering cost component. A long-standing problem is that of determining the structure of the optimal control policy when this system is subject to order quantity capacity constraints; to date, only partial characterisations of the optimal policy have been discussed. An open question is whether a policy with a single continuous interval over which ordering is prescribed is optimal for this problem. Under the so-called ``continuous order property'' conjecture, we show that the optimal policy takes the modified multi-$(s,S)$ form. Moreover, we provide a numerical counterexample in which the continuous order property is violated, and hence  show that a modified multi-$(s,S)$ policy is not optimal in general. However, in an extensive computational study, we show that instances violating the continuous order property are extremely rare in practice, and that the plans generated by a modified multi-$(s,S)$ policy can therefore be considered, for all practical purposes, optimal. Finally, we show that a modified $(s,S)$ policy also performs well in practice.

\end{abstract}

\section{Introduction}

This study focuses on one of the fundamental models in inventory management \citep{Arrow1951,porteus2002foundations}: the periodic review single-item single-stocking location stochastic inventory system under nonstationary demand, complete backorders, and a fixed ordering cost component.
By introducing the concept of $K$-convexity, \cite{Scarf1960} proved, under mild assumptions, that the optimal control policy takes the well-known ($s,S$) form: if the inventory level falls below the reorder point $s$, one should place an order and raise inventory up to level $S$; otherwise, one should not order. 
Compared to the case investigated by \citeauthor{Scarf1960}, in which the order quantity is unconstrained, the capacitated version of the stochastic inventory problem is inherently harder, both structurally and computationally. This work is concerned with this variant of the problem.

If the fixed ordering cost is absent, but ordering capacity constraints are enforced, a so-called {\em modified base stock policy} is optimal for both the finite and infinite horizon cases \citep{Federgruen1986a,Federgruen1986b}. 
While in a classical base stock policy one simply orders up to $S$, in a modified base stock policy, when the inventory level falls below $S$, one should order up to $S$, or as close to $S$ as possible, given the ordering capacity. The classical base stock policy is thus ``modified'' to embed order saturation.

In the presence of a positive fixed ordering cost, \cite{Wijngaard1972} was the first to investigate the influence of capacity constraints on the structure of the optimal control policy. In analogy to the aforementioned modified base stock policy, \citeauthor{Wijngaard1972} conjectured that an optimal strategy may feature a so-called {\em modified} $(s,S)$ {\em structure}: if the inventory level is greater or equal to $s$, do not order; otherwise, order up to $S$, or as close to $S$ as possible, given the ordering capacity. 
Unfortunately, both \cite{Wijngaard1972} and \cite{Chen1996X} provided counterexamples that ruled out the optimality of a modified $(s,S)$ policy. However, \cite{Chen1996X} proved that, under stationary demand and a finite horizon, the optimal policy features a so-called $X-Y$ band structure: when initial inventory level is below $X$, it is optimal to order at full capacity; when initial inventory level is above $Y$, it is optimal to not order. 
\cite{Gallego2000Capacitated} introduced $CK$-convexity, a generalisation of Scarf's $K$-convexity; by leveraging this property, they extended the analysis in \citep{Chen1996X} and further characterized the optimal policy by identifying four regions: in two of these regions the optimal policy is completely specified, while it is only partially specified in other two regions. 
\cite{Chan2003A} discussed further properties of the optimal order policy when the inventory level falls within \citeauthor{Chen1996X}'s $X-Y$ band, and devised an efficient algorithm to compute optimal policy parameters. 
\cite{Chen2004The} extended the analysis in \citep{Chen1996X} and proved that the optimal policy continues to exhibit the $X-Y$ band structure under infinite horizon; moreover, \citeauthor{Chen2004The} proved that the $X-Y$ band width is no more than the capacity. 
\cite{Gallego2004} investigated the case in which the fixed ordering cost is large relative to the variable cost of a full order; this assumption allowed them to restrict their analysis to full-capacity orders; under this setting they showed that the optimal policy is a threshold policy: if the inventory level falls below the threshold $s$, issue a full-capacity order; otherwise, do not order. 
Finally, \cite{Shi2014Approximation} developed an approximation algorithm with worst-case performance guarantee. 

As mentioned in \citep{Shi2014Approximation}, when order quantity capacity constraints are enforced, {\em only some partial characterization of the structure of the optimal control policy is available} in the literature. To the best of our knowledge, the problem of determining the structure of the optimal policy of the capacitated stochastic inventory problem remains open. A long-standing open question in the literature, originally posed by \cite{Gallego2000Capacitated}, is whether a policy with a single continuous interval over which ordering is prescribed is optimal for this problem. This is the so-called ``continuous order property'' conjecture, which was later also investigated by \cite{Chan2003A}. To the best of our knowledge, to date this conjecture has never been confirmed or disproved. This gap motivates the present study. 

We make the following contributions to the literature on stochastic inventory control.
\begin{itemize}
\item In light of the results presented in \citep{Chen2004The}, we show how to simplify the optimal policy structure presented by \cite{Gallego2000Capacitated}. Moreover, we extend the discussion in \citep{Gallego2000Capacitated} and provide a full characterisation of the optimal policy for instances for which the continuous order property holds. In particular, we show that the optimal policy takes the {\em modified multi-}$(s,S)$ {\em form}. 
\item We provide a numerical counterexample in which the continuous order property is violated. This closes a fundamental and long standing question in the literature: a policy with a single continuous interval over which ordering is prescribed is not optimal in general. Since generating similar counterexamples is far from trivial, in our Appendix we illustrate the analytical insights we relied upon to generate such rare instances.
\item In an extensive computational study, we show that instances violating the continuous order property are extremely rare in practice, and that a modified multi-$(s,S)$ ordering policy can therefore be considered, for all practical purposes, optimal. Moreover, the number of reorder-point/order-up-to-level pairs that this policy features in each period is generally very low --- i.e. less or equal to 6 in our study --- this means that operating the policy in practice will not result too cumbersome for a manager. Finally, we show that a well-known heuristic policy which has been known for decades, the modified ($s,S$) policy \citep{Wijngaard1972}, also performs well in practice.
\end{itemize}

The rest of this paper is organised as follows. 
In Section \ref{sec:stochastic_lot_sizing}, we introduce the well-known stochastic inventory problem as originally discussed in \cite{Scarf1960}. 
In Section \ref{sec:capacitated_stochastic_lot_sizing}, we extend the problem description to accommodate order quantity capacity constraints. 
In Section \ref{sec:known_properties} we summarise known properties of the optimal policy from the literature. 
In Section \ref{sec:modified_multi_ss} we introduce the so-called ``continuous order property,'' which has been previously conjectured in the literature, and illustrate the structure that the optimal policy would take if this property were to hold.
In Section \ref{sec:counterexample} we present a numerical counterexample in which the continuous order property is violated.
In Section \ref{sec:computational_study} we illustrate results of our extensive computational study aimed at showing that instances violating the continuous order property are extremely rare in practice, that a modified multi-$(s,S)$ ordering policy is optimal from a practical standpoint, and that a modified $(s,S)$ ordering policy is near-optimal.
In Section \ref{sec:conclusions} we draw conclusions. 

\section{Preliminaries on the ($s,S$) policy}\label{sec:stochastic_lot_sizing}

The rest of this work is concerned with a single-item single-stocking point inventory control problem. A finite planning horizon of $n$ discrete time periods, which are labelled in reverse order for convenience, is assumed. Period demands are stochastic, $d_t$ in period $t$, with known probability density and cumulative distribution functions $f_t$ and $F_t$, respectively. The cost components that are taken into account include: the ordering cost $c(x)$ for placing an order for $x$ units; the inventory holding cost $h$ for any excess unit of stock carried over to next period; and the shortage cost $p$ that is incurred for each unit of unmet demand in any given period. Unmet demand is backordered. Without loss of generality, it is assumed that there is no lead-time and deliveries are instantaneous.

Let $x$ represent the pre-order inventory level, and $\widehat{C}_n(x)$ denote the minimum expected total cost achieved by employing an optimal replenishment policy over the planning horizon $n,\ldots,1$; then one can write 
\[
\widehat{C}_n(x)\triangleq\min_{x\leq y}\left\{c(y-x)+L_n(y)+\int_0^\infty \widehat{C}_{n-1}(y-\xi)f_n(\xi) \mathrm{d}\xi \right\},
\]
where $\widehat{C}_0\triangleq0$ and
$L_n(y)\triangleq\int_0^y h(y-\xi)f_n(\xi) \mathrm{d}\xi + \int_y^\infty p(\xi-y)f_n(\xi) \mathrm{d}\xi$.

Following \cite{Scarf1960}, it is assumed that the ordering cost takes the form
\[
c(x)\triangleq\left\{
\begin{array}{lr}
0			&\quad x=0,\\
K+vx			& \quad x>0.
\end{array}
\right.
\]

For convex $L_n(y)$, \cite{Scarf1960} proved that the optimal policy takes the $(s,S)$ form, and thus features two policy control parameters: $s$ and $S$. In the $(s,S)$ policy, an order of size $S-x$ is placed if and only if the pre-order inventory level is $x<s$. 

More specifically, \cite{Scarf1960} introduced the concept of $K$-convexity (Definition \ref{def:k_convex}).
\begin{definition}[$K$-convexity]\label{def:k_convex}
Let $K\geq 0$, $g(x)$ is $K$-convex if for all $x$, $a>0$, and $b>0$,
\[(K+g(x+a)-g(x))/a\geq (g(x)-g(x-b))/b;\]
\end{definition}
and proved that $\widehat{G}_n(y)$ is $K$-convex, where
\[
\widehat{G}_n(y)\triangleq vy+L_n(y)+\int_0^\infty \widehat{C}_{n-1}(y-\xi)f_n(\xi) \mathrm{d}\xi.
\]
This observation implies that the $(s,S)$ policy is optimal, and the policy parameters $s$ and $S$ satisfy
$\widehat{G}_n(s)=\widehat{G}_n(S)+K$.
Note that when the order quantity is not subject to capacity constraints, $S$ incidentally coincides with the global minimizer of $\widehat{G}_n(y)$. In what follows, we will see that this may not be the case when a capacity constraint is enforced on the order quantity.

\section{Capacitated ordering}\label{sec:capacitated_stochastic_lot_sizing}

The stochastic inventory problem investigated in \citep{Scarf1960} assumes that order quantity $Q$ in each period can be as large as needed. In practice, one may want to impose the restriction that $0\leq Q \leq B$, where $B$ is a positive value denoting the maximum order quantity in each period. 

We generalise $\widehat{C}_n(x)$ and $\widehat{G}_n(x)$  to reflect capacity restrictions
\begin{equation}\label{eq:Cn}
C_n(x)\triangleq \min_{x\leq y \leq x+B}\left\{c(y-x)+L_n(y)+\int_0^\infty C_{n-1}(y-\xi)f_n(\xi) \mathrm{d}\xi \right\};
\end{equation}
\begin{equation}\label{eq:Gn}
G_n(y)\triangleq vy+L_n(y)+\int_0^\infty C_{n-1}(y-\xi)f_n(\xi) \mathrm{d}\xi.
\end{equation}
Finally, we present a useful result that will be used in the coming sections.
\begin{lemma}\label{lemma:limiting_behaviour}
$G_n(x)$ is coercive.
\end{lemma}
\begin{proof}
The limiting behaviour of $G_n(x)$ can be characterized as $\lim_{x\rightarrow \infty}  G'_n(x) = nh$ and 
$\lim_{x\rightarrow -\infty} G'_n(x) = -np$, 
and from the fundamental theorem of calculus it follows that $G_n(\infty)=G_n(-\infty)=\infty$.
\end{proof}

\section{Review of known properties of the optimal policy}\label{sec:known_properties}

We next introduce\footnote{This was originally called strong $CK$-convexity in \citep{Gallego2000Capacitated}; however, in line with \cite{Scarf1960}, in the present work we used letter $C$ to denote the cost function and we adopted letter $B$ for the ordering capacity, hence the concept has been renamed $(K,B)$-convexity.} ``$(K,B)$-convexity (i)'' for a function $g$ \citep{Gallego2000Capacitated}.
\begin{definition}\label{def:KB_convex_i}
Let $K\geq 0$, $B\geq 0$, $g$ is $(K,B)$-convex (i) if it satisfies
\[(K+g(x+a)-g(x))/a\geq (g(y)-g(y-b))/b\]
for $0<a\leq B$, $0<b\leq B$, and $y\leq x$.
\end{definition}
\begin{example}\label{sec:numerical_example}
Consider a planning horizon of $n=4$ periods, and a demand $d_t$ distributed in each period $t=1,\ldots,n$ according to a Poisson law with rate $\lambda_t\in\{20, 40, 60, 40\}$. Other problem parameters are $K=100$, $h=1$ and $p=10$; to better conceptualise the example we let $v=0$. 
In Fig. \ref{fig:KB_convexity_i} we plot $G_n(y)$ and illustrate the concept of $(K,B)$-convexity (i) for the case in which $B=65$. 
\end{example}
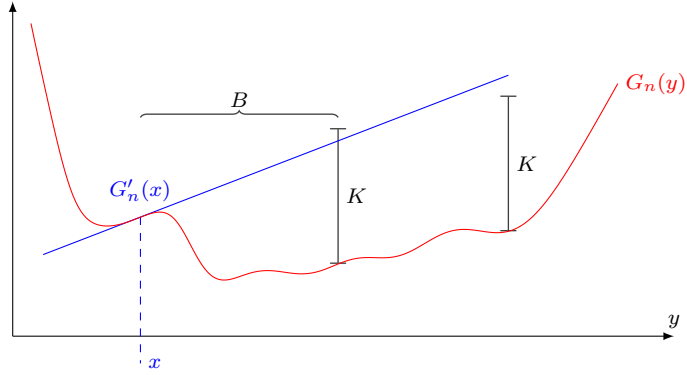
\begin{figure}
\centering
\begin{tikzpicture}[x=0.04cm, y=0.018cm]
\footnotesize
\draw [-latex] ([xshift=-0mm] 0.0,0) -- ([xshift=3mm] 210.0,0) node[above] {$y$};
\draw [-latex] ([yshift=-0mm] 0,0.0) -- ([yshift=3mm] 0, 230.0);


\draw[|-|,color=darkgray,line width=0.5pt]
  (107,153.2702526094082) -- (107,53.2702526094082) node[midway,black,right=0pt] {$K$};

\draw[|-|,color=darkgray,line width=0.5pt]
  (163,177.21922694362269) -- (163,77.21922694362269) node[midway,black,right=0pt] {$K$};

\draw[color=darkgray,decoration={brace,raise=4pt},decorate,line width=0.5pt]
  (42,153.2702526094082) -- (107,153.2702526094082) node[midway,black,above=5pt] {$B$};
  
\draw [color=blue, mark= , style=dashed] (42,87.72015201566262) node[above = 3pt] {$G'_n(x)$} -- (42,-20) node[right] {$x$};
\draw [color=blue, mark= , style=solid] (10.00, 60.08752987723638) -- (163,192.20600447658684); 

\draw [smooth, color=red, mark= , style=solid] plot coordinates {%
(6, 480.3525444206225-250)
(7, 469.75598324241037-250)
(8, 459.3240096288101-250)
(9, 449.01568691325707-250)
(10, 438.8041602405519-250)
(11, 428.68152457586-250)
(12, 418.6644076922001-250)
(13, 408.7987852450556-250)
(14, 399.1619544271999-250)
(15, 389.8597124714496-250)
(16, 381.017775549053-250)
(17, 372.768088220551-250)
(18, 365.23233430376825-250)
(19, 358.50601169858106-250)
(20, 352.64645366403465-250)
(21, 347.6671653736948-250)
(22, 343.53919316784004-250)
(23, 340.19854663951196-250)
(24, 337.557478789343-250)
(25, 335.5169733477567-250)
(26, 333.97807393615176-250)
(27, 332.8504254251565-250)
(28, 332.057388695932-250)
(29, 331.53786472419876-250)
(30, 331.2455201090355-250)
(31, 331.14634909007-250)
(32, 331.2154930425052-250)
(33, 331.43406181531134-250)
(34, 331.7864509284346-250)
(35, 332.2584026013554-250)
(36, 332.83580633961583-250)
(37, 333.5040540206645-250)
(38, 334.2478503260832-250)
(39, 335.0511195815081-250)
(40, 335.89625427275814-250)
(41, 336.76535427753254-250)
(42, 337.6386781274589-250)
(43, 338.4939783671208-250)
(44, 339.3054286535175-250)
(45, 340.0420005041289-250)
(46, 340.66526860009515-250)
(47, 341.1269818232568-250)
(48, 341.36725851261565-250)
(49, 341.3147003855667-250)
(50, 340.88974599908585-250)
(51, 340.0119900488247-250)
(52, 338.611047532022-250)
(53, 336.6392173353053-250)
(54, 334.08323044401476-250)
(55, 330.9721898541786-250)
(56, 327.37954856214026-250)
(57, 323.41838803034415-250)
(58, 319.23085653618375-250)
(59, 314.97388134391207-250)
(60, 310.803886821463-250)
(61, 306.8630656940294-250)
(62, 303.26904170171633-250)
(63, 300.10878638117407-250)
(64, 297.436726859719-250)
(65, 295.27631278596124-250)
(66, 293.623975616665-250)
(67, 292.45437920941106-250)
(68, 291.7260653724512-250)
(69, 291.3866958259477-250)
(70, 291.3776992183416-250)
(71, 291.63804946489637-250)
(72, 292.10718201486975-250)
(73, 292.7271506781469-250)
(74, 293.4441560207945-250)
(75, 294.2098042186789-250)
(76, 294.9815059513597-250)
(77, 295.72246720943934-250)
(78, 296.4024410599451-250)
(79, 296.997698595793-250)
(80, 297.49090691925034-250)
(81, 297.870920645062-250)
(82, 298.13249013986683-250)
(83, 298.275888081496-250)
(84, 298.30645590298934-250)
(85, 298.23407304183456-250)
(86, 298.07255426504594-250)
(87, 297.83898323707405-250)
(88, 297.552993487089-250)
(89, 297.2360106109957-250)
(90, 296.91047160070946-250)
(91, 296.59903844165456-250)
(92, 296.32382351166495-250)
(93, 296.1056439421127-250)
(94, 295.963321178641-250)
(95, 295.91304080555005-250)
(96, 295.96778661218883-250)
(97, 296.1368621881501-250)
(98, 296.4255132336748-250)
(99, 296.8346642718011-250)
(100, 297.360784310719-250)
(101, 297.9958967209416-250)
(102, 298.72774842432995-250)
(103, 299.5401515802864-250)
(104, 300.4135064864436-250)
(105, 301.32550683100277-250)
(106, 302.25201767030165-250)
(107, 303.1681028869593-250)
(108, 304.04916409145505-250)
(109, 304.87213785992265-250)
(110, 305.61668567321834-250)
(111, 306.26630266031367-250)
(112, 306.8092690981413-250)
(113, 307.23937366949434-250)
(114, 307.55634989526493-250)
(115, 307.76598609671214-250)
(116, 307.8798929242094-250)
(117, 307.91493843368414-250)
(118, 307.8923860860348-250)
(119, 307.83679314665795-250)
(120, 307.77474349309165-250)
(121, 307.73349828382834-250)
(122, 307.73964971626515-250)
(123, 307.8178575668899-250)
(124, 307.98973653918966-250)
(125, 308.2729464054667-250)
(126, 308.6805185752897-250)
(127, 309.2204340672018-250)
(128, 309.89545073987375-250)
(129, 310.7031633348688-250)
(130, 311.6362692215891-250)
(131, 312.6830059660491-250)
(132, 313.8277237527019-250)
(133, 315.0515557046083-250)
(134, 316.3331514968708-250)
(135, 317.6494435022018-250)
(136, 318.97641926079274-250)
(137, 320.2898786801422-250)
(138, 321.5661585732077-250)
(139, 322.78281065157273-250)
(140, 323.9192217948307-250)
(141, 324.9571673563599-250)
(142, 325.8812895909029-250)
(143, 326.67949422682335-250)
(144, 327.34325901255886-250)
(145, 327.86784898765313-250)
(146, 328.2524344591506-250)
(147, 328.5001093233423-250)
(148, 328.61780949209543-250)
(149, 328.6161337087141-250)
(150, 328.5090718464008-250)
(151, 328.3136487034839-250)
(152, 328.04949414286364-250)
(153, 327.7383529766957-250)
(154, 327.4035500907345-250)
(155, 327.0694277895308-250)
(156, 326.7607731223718-250)
(157, 326.50225296797277-250)
(158, 326.31787391250947-250)
(159, 326.2304824988906-250)
(160, 326.26131934710395-250)
(161, 326.4296380699043-250)
(162, 326.75239699355745-250)
(163, 327.24402863171076-250)
(164, 327.91628865169105-250)
(165, 328.77818319723093-250)
(166, 329.83597068748384-250)
(167, 331.093231890589-250)
(168, 332.5510002055803-250)
(169, 334.207942724825-250)
(170, 336.06058180447627-250)
(171, 338.10354652765193-250)
(172, 340.3298435645582-250)
(173, 342.73113745661396-250)
(174, 345.29803120534774-250)
(175, 348.0203391514997-250)
(176, 350.8873454034014-250)
(177, 353.88804243727304-250)
(178, 357.01134587334553-250)
(179, 360.24628276824916-250)
(180, 363.582152005108-250)
(181, 367.00865647001586-250)
(182, 370.5160076512769-250)
(183, 374.095004072049-250)
(184, 377.73708556432763-250)
(185, 381.4343658178782-250)
(186, 385.1796459040432-250)
(187, 388.9664115986113-250)
(188, 392.78881733069113-250)
(189, 396.6416594878788-250)
(190, 400.52034163429386-250)
(191, 404.4208339686314-250)
(192, 408.33962908388514-250)
(193, 412.2736958060818-250)
(194, 416.2204326008385-250)
(195, 420.17762175561023-250)
(196, 424.143385281167-250)
(197, 428.11614323458696-250)
(198, 432.0945749520266-250)
(199, 436.07758349482-250)} 
 node[right] {$G_n(y)$};
  
\end{tikzpicture}
\caption{$(K,B)$-convexity (i) in the context of Example \ref{sec:numerical_example}, when $B=65$. For the sake of illustration, we assume $G_n$ differentiable, set $x=y$, and let $b\rightarrow 0$, obtaining $K+G_n(x+a)-G_n(x)-aG_n'(x)\geq 0$ for $0<a\leq B$.}
\label{fig:KB_convexity_i}
\end{figure} 

\begin{lemma}\label{lemma:decreasing}
If $G_n$ (resp. $C_n$) is $(K,B)$-convex (i) and it is optimal to place an order at $x_0$, then $G_n(y)$ (resp. $C_n(y)$) is nonincreasing for $y\leq x_0$.
\end{lemma}
\begin{proof}
Since $G_n$ is $(K,B)$-convex (i), if it is optimal to place an order at $x_0$, say an order of $a$ units, then $0\geq(K+G_n(x_0+a)-G_n(x_0))/a$, and $G_n$ is nonincreasing for $y\leq x_0$, since 
$0\geq(K+G_n(x_0+a)-G_n(x_0))/a\geq (G_n(y)-G_n(y-b))/b$, 
for $y\leq x_0$ and $0<b\leq B$. The proof for $C_n$ is identical.
\end{proof}

\begin{lemma}\label{lemma:s_m}
If $G_n$ is $(K,B)$-convex (i), there exists a pair of values $S_m$ and $s_m$ such that
$s_m\triangleq\sup\{x|C_n(x)=G_n(x)-vx\}$ 
is the maximum inventory level at which it is optimal to place an order, 
and $S_m\triangleq s_m+a$, where $0<a\leq B$ is the order quantity at $s_m$.
\end{lemma}
\begin{proof}
Let $x_0$ be any point at which it is optimal to order, say, $a$ units, $0< a \leq B$. 
$G_n(y)$ is a nonincreasing function for $y\leq x_0$ (Lemma \ref{lemma:decreasing}).
This result implies that there must exist an upper bound on inventory level beyond which no ordering is optimal. 
Otherwise $G_n(y)$ would be a nonincreasing function for all $y$, which contradicts Lemma \ref{lemma:limiting_behaviour}. 
\end{proof}

\begin{definition}\label{def:KB_convex_ii}
Let $K\geq 0$, $B\geq 0$, $g$ is $(K,B)$-convex (ii) if it satisfies
\[(K+g(x+a)-g(x))/a\geq (K+g(y)-g(y-B))/B\]
for $0<a\leq B$ and $y\leq x$.
\end{definition}

\begin{example}
In Fig. \ref{fig:KB_convexity_ii} we plot $C_n(y)$ and illustrate the concept of $(K,B)$-convexity (ii) for our numerical example. 
\end{example}
\begin{figure}
\centering
\begin{tikzpicture}[x=0.0380952380952381cm, y=0.018cm]
\footnotesize
\draw [-latex] ([xshift=-0mm] 0.0,0) -- ([xshift=3mm] 210.0,0) node[above] {$y$};
\draw [-latex] ([yshift=-0mm] 0,0.0) -- ([yshift=3mm] 0, 270.0);

\draw [color=blue, mark= , style=solid] (60,61.363152308511985+100) -- (-5,410.8030113900325-250) node[midway, above] {$Z$};
\draw [color=blue, mark= , style=dashed] (-5,410.8030113900325-250) -- (-5,-20)  node[below] {$y-B$};
\draw [color=blue, mark= , style=dashed] (60,61.363152308511985+100) -- (60,-20) node[below] {$x$};
\draw[color=darkgray,decoration={brace,raise=4pt},decorate,line width=0.5pt]
  (60,61.363152308511985+100) -- (60,61.363152308511985) node[midway,black,right=5pt] {$K$};

\draw [color=blue, mark= , style=solid] (60,61.363152308511985) -- (110,55.72807330424445+100) node[midway, above] {$X$};
\draw [color=blue, mark= , style=dashed] (125,0) -- (125,-20) node[below, right] {$x+B$};
\draw [color=blue, mark= , style=dashed] (110,55.72807330424445+100) -- (110,-20) node[below] {$x+a$};
\draw[color=darkgray,decoration={brace,raise=4pt},decorate,line width=0.5pt]
(60,0) -- (125,0) node[midway,black,above=5pt] {$B$};
\draw[color=darkgray,decoration={brace,raise=4pt},decorate,line width=0.5pt]
  (110,55.72807330424445+100) -- (110,55.72807330424445) node[midway,black,right=5pt] {$K$};

\draw [smooth, color=red, mark= , style=solid] plot coordinates {%
(-50, 489.85883704001924-250)
(-49, 481.0169001176226-250)
(-48, 472.76721278912066-250)
(-47, 465.231458872338-250)
(-46, 458.50513626715053-250)
(-45, 452.6455782326042-250)
(-44, 447.6662899422644-250)
(-43, 443.5383177364097-250)
(-42, 440.19767120808143-250)
(-41, 437.55660335791254-250)
(-40, 435.5160979163262-250)
(-39, 433.9771985047214-250)
(-38, 432.8495499937262-250)
(-37, 432.05651326450163-250)
(-36, 431.5369892927682-250)
(-35, 431.24464467760527-250)
(-34, 431.1454736586394-250)
(-33, 431.1454736586394-250)
(-32, 431.1454736586394-250)
(-31, 431.1454736586394-250)
(-30, 431.1454736586394-250)
(-29, 431.1454736586394-250)
(-28, 431.1454736586394-250)
(-27, 431.1454736586394-250)
(-26, 431.1454736586394-250)
(-25, 431.1454736586394-250)
(-24, 431.1454736586394-250)
(-23, 431.1454736586394-250)
(-22, 431.1454736586394-250)
(-21, 431.1454736586394-250)
(-20, 431.1454736586394-250)
(-19, 431.1454736586394-250)
(-18, 431.1454736586394-250)
(-17, 431.1454736586394-250)
(-16, 431.1454736586394-250)
(-15, 431.1454736586394-250)
(-14, 431.1454736586394-250)
(-13, 431.1454736586394-250)
(-12, 431.1454736586394-250)
(-11, 431.1454736586394-250)
(-10, 430.97131442274815-250)
(-9, 427.37867313070984-250)
(-8, 423.41751259891373-250)
(-7, 419.22998110475334-250)
(-6, 414.9730059124816-250)
(-5, 410.8030113900325-250)
(-4, 406.8621902625991-250)
(-3, 403.2681662702858-250)
(-2, 400.1079109497436-250)
(-1, 397.4358514282887-250)
(0, 395.2754373545309-250)
(1, 393.62310018523453-250)
(2, 392.4535037779807-250)
(3, 391.72518994102103-250)
(4, 391.38582039451734-250)
(5, 391.37682378691125-250)
(6, 391.37682378691125-250)
(7, 391.37682378691125-250)
(8, 391.37682378691125-250)
(9, 391.37682378691125-250)
(10, 391.37682378691125-250)
(11, 391.37682378691125-250)
(12, 391.37682378691125-250)
(13, 391.37682378691125-250)
(14, 391.37682378691125-250)
(15.0,140.43658281254795)
(16.0,131.59182294448976)
(17.0,123.34485708702891)
(18.0,115.82195421134861)
(19.0,109.11577505138189)
(20.0,103.28221166344383)
(21.0,98.3341748876141)
(22.0,94.15193821388169)
(23.0,90.76077359729624)
(24.0,88.07211646600575)
(25.0,85.98623551985622)
(26.0,84.40368033057541)
(27.0,83.23397537669098)
(28.0,82.4005562763046)
(29.0,81.84236546244676)
(30.0,81.51330975550985)
(31.0,81.37974902174454)
(32.0,81.41722635447968)
(33.0,81.60720390245467)
(34.0,81.9343222323439)
(35.0,82.38444836651428)
(36.0,82.9435612469785)
(37.0,83.5973674787146)
(38.0,84.33147269399126)
(39.0,85.12887827468268)
(40.0,85.97221873910905)
(41.0,86.84266233513102)
(42.0,87.72015201566262)
(43.0,88.58367145748844)
(44.0,89.41142017055427)
(45.0,90.18176935331286)
(46.0,90.85546268250533)
(47.0,91.36366561082264)
(48.0,91.6471355994118)
(49.0,91.63505345992775)
(50.0,91.25278676921153)
(51.0,90.4175526098773)
(52.0,89.05575783551762)
(53.0,87.11702510257237)
(54.0,84.58443474929794)
(55.0,81.48431027241156)
(56.0,77.88734042371385)
(57.0,73.9315141269447)
(58.0,69.78240163090493)
(59.0,65.53986509669261)
(60.0,61.363152308511985)
(61.0,57.39841366230439)
(62.0,53.76791976775513)
(63.0,50.56438679623767)
(64.0,47.84599380210955)
(65.0,45.64027648349963)
(66.0,43.946981191564305)
(67.0,42.74324554564083)
(68.0,41.98917430789004)
(69.0,41.63318364229559)
(70.0,41.61676878905865)
(71.0,41.87843667884266)
(72.0,42.35672517600551)
(73.0,42.992479645067135)
(74.0,43.73051950726665)
(75.0,44.52077186964033)
(76.0,45.319014416884045)
(77.0,46.08731553998405)
(78.0,46.79428377067808)
(79.0,47.415106665322185)
(80.0,47.93140616093086)
(81.0,48.331134638544654)
(82.0,48.60821538933442)
(83.0,48.76233398803828)
(84.0,48.79842773603093)
(85.0,48.7261037868214)
(86.0,48.559073467751944)
(87.0,48.31448416759156)
(88.0,48.012202179171936)
(89.0,47.674033362889304)
(90.0,47.32296258945496)
(91.0,46.98237336923819)
(92.0,46.675289413074495)
(93.0,46.423509910330154)
(94.0,46.24693467064117)
(95.0,46.1628760609745)
(96.0,46.18543528321214)
(97.0,46.324950254167504)
(98.0,46.58753512228412)
(99.0,46.97475003524954)
(100.0,47.48341947086499)
(101.0,48.10562113216395)
(102.0,48.82886173683005)
(103.0,49.63647243433411)
(104.0,50.508368371058)
(105.0,51.42144836798229)
(106.0,52.35069031877691)
(107.0,53.2702526094082)
(108.0,54.15496425652225)
(109.0,54.981157069755284)
(110.0,55.72807330424445)
(111.0,56.37900006817517)
(112.0,56.92221197951869)
(113.0,57.351657457584224)
(114.0,57.66732206324622)
(115.0,57.87526548027364)
(116.0,57.98732422043821)
(117.0,58.02050937337924)
(118.0,57.99609246740164)
(119.0,57.93821860856275)
(120.0,57.87400131346857)
(121.0,57.83078211092487)
(122.0,57.83534920184974)
(123.0,57.91270039370164)
(124.0,58.084922415226686)
(125.0,58.37027068032131)
(126.0,58.78242179309018)
(127.0,59.32999554088576)
(128.0,60.01631494718964)
(129.0,60.839406543987195)
(130.0,61.7922194220626)
(131.0,62.86303282949325)
(132.0,64.03601655616711)
(133.0,65.29190626860958)
(134.0,66.60875693718174)
(135.0,67.96273739619681)
(136.0,69.32894805218638)
(137.0,70.68222018558032)
(138.0,71.99788336159708)
(139.0,73.25249193107231)
(140.0,74.4244824934903)
(141.0,75.49476097599921)
(142.0,76.4472010227932)
(143.0,77.2690623188609)
(144.0,77.95130467561052)
(145.0,78.48880364044362)
(146.0,78.88046288077305)
(147.0,79.12922263293416)
(148.0,79.24196589238221)
(149.0,79.22932599165478)
(150.0,79.10540198819592)
(151.0,78.88739086081716)
(152.0,78.5951478285973)
(153.0,78.25069421309212)
(154.0,77.87765851266397)
(155.0,77.50072805787573)
(156.0,77.14506732501223)
(157.0,76.83575099793103)
(158.0,76.59722200681205)
(159.0,76.45278954940835)
(160.0,76.42418012159698)
(161.0,76.53115214809009)
(162.0,76.79118208303214)
(163.0,77.21922694362269)
(164.0,77.82756529902366)
(165.0,78.62571596556865)
(166.0,79.62043062396896)
(167.0,80.81575532537232)
(168.0,82.21315258252622)
(169.0,83.81167536679283)
(170.0,85.60818298599202)
(171.0,87.59758844443564)
(172.0,89.77312690021091)
(173.0,92.12663525537243)
(174.0,94.64883368174992)
(175.0,97.32960088788855)
(176.0,100.15823620896003)
(177.0,103.12370285826063)
(178.0,106.21484808237648)
(179.0,109.42059728703805)
(180.0,112.73012045521705)
(181.0,116.13297031156031)
(182.0,119.61919266879426)
(183.0,123.17941020504657)
(184.0,126.80488155955408)
(185.0,130.48753810055928)
(186.0,134.2200010234559)
(187.0,137.99558159463612)
(188.0,141.80826738782497)
(189.0,145.65269728406724)
(190.0,149.52412784894767)
(191.0,153.41839348096863)
(192.0,157.33186246424168)
(193.0,161.2613907747371)
(194.0,165.20427519780236)
(195.0,169.15820702822384)
(196.0,173.1212273526869)
(197.0,177.09168466526467)
(198.0,181.06819534434084)
(199.0,185.04960732672913)
(200.0,189.03496715249912)} 
 node[right] {$C_n(y)$};
  
\end{tikzpicture}
\caption{$(K,B)$-convexity (ii) in the context of Example \ref{sec:numerical_example}, when $B=65$. For the sake of illustration, we set $x=y$. Intuitively, for all $a\leq B$, the slope of segment $X$ is greater or equal to the slope of segment $Z$.}
\label{fig:KB_convexity_ii}
\end{figure}
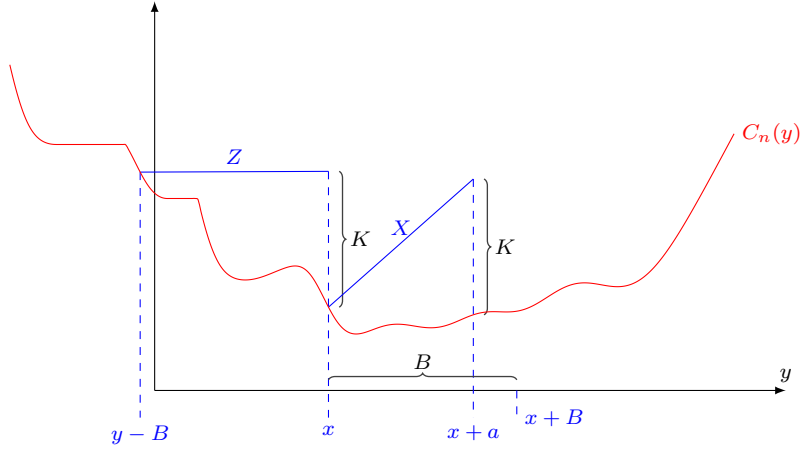 

\begin{definition}\label{def:KB_convex}
$g$ is $(K,B)$-convex if it satisfies $(K,B)$-convex (i) and $(K,B)$-convex (ii).
\end{definition}

\begin{theorem}\label{thm:kb_convexity_Gn_Cn}
$G_n(x)$ and $C_n(x)$ are $(K,B)$-convex.
\end{theorem}
\begin{proof}
Proofs are available in \citep{Gallego2000Capacitated,Chen2004The}.  
\end{proof}


Originally in \citep{Chen1996X}, and then by introducing the concept of $(K,B)$-convexity (ii) in \citep{Chen2004The}, \citeauthor{Chen1996X} established existence of a level $Y\triangleq s_m$ beyond which it is not optimal order, and of another level $X\triangleq Y-B$ below which it is optimal to order up to capacity. The optimal policy therefore features a so-called ``$X-Y$ band'' structure.

\begin{lemma}\label{lemma:s_m-B}
If $G_n$ is $(K,B)$-convex (ii), it is optimal to order up to capacity at any $y\leq s_m-B$.
\end{lemma}
\begin{proof}
Let $x_0$ be any point at which it is optimal to order something. By $(K,B)$-convexity (ii), 
\[0>(K+G_n(x_0+a)-G_n(x_0))/a\geq (K+G_n(y)-G_n(y-B))/B,\]
for all $y\leq x_0$. Thus, $0>K+G_n(y)-G_n(y-B)$, because $G_n$ is nonincreasing for $y\leq x_0$. Hence, it is optimal to order up to capacity at any $y\leq  x_0 - B$.
\end{proof}


\cite{Gallego2000Capacitated} further characterised the structure of the optimal policy within \citeauthor{Chen1996X}'s $X-Y$ band. In particular, they showed that 
\begin{equation}\label{eq:Cn_policy}
C_n(x)=\left\{
\begin{array}{lrlr}
G^B_n(x)										&\quad x< \min\{s'-B,s\}					\\[-4pt]
\alpha\min\{-vx+G_n(x),G^B_n(x)\}+(1-\alpha)G^S_n(x)	&\quad \min\{s'-B,s\} \leq x < \max\{s'-B,s\}   	\\[-4pt]
\min\{-vx+G_n(x),G^S_n(x)\}						&\quad \max\{s'-B,s\} \leq x \leq s'			\\[-4pt]
-vx+G_n(x)									&\quad x> s'							
\end{array}
\right.
\end{equation}
where
\[
\begin{array}{l}
G^B_n(x)\triangleq K-vx+G_n(x+B)							\\
G^S_n(x)\triangleq K-vx+\min_{x\leq y \leq x+B} G_n(y) 			\\
s\triangleq \inf\{x|K+\min_{x\leq y \leq x+B} G_n(y) - G_n(x)\geq 0\}	\\
s'\triangleq \max\{x\leq S_m|K+\min_{x\leq y \leq x+B} G_n(y) - G_n(x)\leq 0\}
\end{array}
\]
and $\alpha$ is an indicator variable that takes value $1$ if $s'-s>B$, and $0$ otherwise.

\begin{lemma}\label{lemma:s}
$s'-B < s \leq s'$
\end{lemma}
\begin{proof}
Observe that $s_m=s'$, thus $s \leq s'$; by Lemma \ref{lemma:s_m-B}, it is optimal to order up to capacity at any $x\leq s_m-B$; 
hence $C_n(x)=G^B_n(x)$ for $x< s'-B$, and $C_n(x)=G^S_n(x)$ at $x=s'-B$; therefore $s> s'-B$.
\end{proof}

By leveraging Lemma \ref{lemma:s}, it is possible to further simplify \citeauthor{Gallego2000Capacitated}'s structure of the optimal policy as follows.

\begin{lemma}
\begin{equation}\label{eq:Cn_policy_simplified}
C_n(x)=\left\{
\begin{array}{lrlr}
G^B_n(x)											&\quad x< s_m-B			\\[-4pt]
G^S_n(x)											&\quad s_m-B \leq x < s     	\\[-4pt]
\min\{-vx+G_n(x),G^S_n(x)\}							&\quad s \leq x \leq s_m		\\[-4pt]
-vx+G_n(x)										&\quad x> s_m							
\end{array}
\right.
\end{equation}
\end{lemma}
\begin{proof}
Observe that $s_m=s'$; because of Lemma \ref{lemma:s}, it is clear that $s_m-s\leq B$ and $\alpha=0$.
\end{proof}

\section{The modified multi-($s,S$) policy}\label{sec:modified_multi_ss}

\begin{definition}[Continuous Order Property]\label{def:always_order}
Let $x_0$ be an inventory level at which it is optimal to place an order, $C_n$ is said to have the continuous order property if it is optimal to place an order at $y$, for all $y<x_0$.
\end{definition}

\begin{lemma}\label{lemma:cop_convex_set}
If $C_n$ has the continuous order property, $\{x|C_n(x)-(G_n(x)-vx)<0\}$ is a convex set.
\end{lemma}
\begin{proof}
If $C_n$ has the continuous order property, in \citeauthor{Gallego2000Capacitated}'s policy $s=s'$; hence for all $x\leq s'$ it is optimal to order, that is $C_n(x)-(G_n(x)-vx)\leq 0$, and for all $x>s'$ it is optimal to not order, that is $C_n(x)-(G_n(x)-vx)> 0$; hence $\{x|C_n(x)-(G_n(x)-vx)<0\}$ is a convex set.
\end{proof}
In Fig. \ref{fig:CnMinusGn} we illustrate Lemma \ref{lemma:cop_convex_set} for our numerical example, which incidentally satisfies the continuous order property.
\begin{figure}
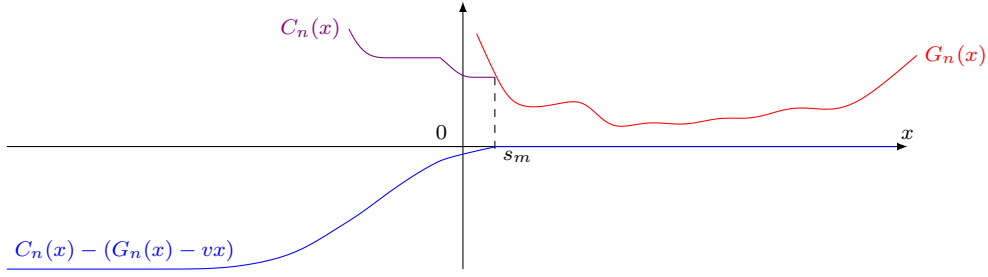

\centering
\include{numerical_example_9a}
\caption{Lemma \ref{lemma:cop_convex_set} in the context of Example \ref{sec:numerical_example}, when $B=65$.}
\label{fig:CnMinusGn}
\end{figure}

Consider $C_n$ as defined in Eq \eqref{eq:Cn}, let this function be $(K,B)$-convex, and assume that the continuous order property holds. When inventory falls below the reorder threshold $s_m$, defined in Lemma \ref{lemma:s_m}, the optimal policy takes the following form: at the beginning of each period, let $x$ be the initial inventory, the order quantity $Q$ is computed as
\begin{equation}\label{eq:optimal_policy_cap}
Q=\left\{
\begin{array}{lr}
\min\{S_k-x,  B\} 	&\quad s_{k-1}< x\leq s_k, \\
0 				&\quad x> s_{m};
\end{array}\right.
\end{equation}
where $k=1,\ldots, m$ and $s_0=-\infty$. In essence, the policy features $m$ reorder thresholds $s_1<s_2<\ldots<s_m$ and associated order-up-to-levels $S_1<S_2<\ldots<S_m$; at the beginning of each period, if inventory drops between reorder threshold $s_k$ and reorder threshold $s_{k-1}$, it is optimal to order a quantity $Q=\min\{S_i-x,  B\}$. For convenience, we denote the case $Q=B$ as {\em saturated ordering}, and the case $0<Q<B$ as {\em unsaturated ordering}. We shall name this control rule {\em modified multi}-$(s, S)$ {\em policy},  or ($s_k$,$S_k$) policy in short. This policy structure was also described in \cite[p. 612]{Gallego2000Capacitated}; however, \citeauthor{Gallego2000Capacitated} did not establish a relation between the continuous order property and the optimality of this policy.

\begin{lemma}\label{lemma:sm_Sm_cost_minimal_maximal}
Consider $S_m$ and $s_m$ as defined in Lemma \ref{lemma:s_m}, and let $S^*\triangleq \arg \min_y G_n(y)$,
\begin{enumerate}[(a)]
\item $S_m\leq S^*$;
\item $G_n(S_m)\leq G_n(x)$ for $x<S_m$;
\item $G_n(s_m)> G_n(x)$ for $s_m<x\leq S_m$.
\end{enumerate}
\end{lemma}
\begin{proof}
(a) If $s_m\geq S^*-B$, then we must necessarily order up to $S^*$ as no point dominates a global minimum.
If $s_m< S^*-B$, then we do not have sufficient capacity to reach $S^*$, hence the optimum order quantity will be a value $a\leq B$; and from $s_m$ we will order up to a point $S_m\triangleq s_m+a\leq S^*$. 
(b) Assume, ex absurdo, $G_n(S_m)> G_n(S)$ for some $S$ such that $s_m<S<S_m$; then from $s_m$ it would not be optimal to order up to $S_m$, which contradicts Lemma \ref{lemma:s_m}. 
(c) Assume, ex absurdo, $G_n(s_m)\leq G_n(s)$ for some $s$ such that $s_m<s\leq S_m$; then from $s$ it would be optimal to order up to $S_m$, this contradicts the fact that $s_m$ is the maximum inventory level at which it is optimal to place an order (Lemma \ref{lemma:s_m}). 
\end{proof}
Observe that $S_m$ is not necessarily a minimizer of $G_n$; this is further illustrated in Appendix \ref{sec:appendix_1}.

By building upon $(K,B)$-convexity of $G_n(x)$ and $C_n(x)$, and upon the assumption that the continuous order property in Definition \ref{def:always_order} holds, we next  establish existence of reorder thresholds $s_1<s_2<\ldots<s_m$ and associated order-up-to-levels $S_1<S_2<\ldots<S_m$ that can be used to control the system according to the optimal ordering policy in Eq. \eqref{eq:optimal_policy_cap}. 

\begin{definition}
A function $g:\mathcal{D}\rightarrow\mathbb{R}$ defined on a convex subset $\mathcal{D}\in\mathbb{R}$ is {\em quasiconvex} if, for all $x,y\in \mathcal{D}$ and $\lambda\in[0,1]$, \[g(\lambda x +(1-\lambda)y)\leq \max\{g(x),g(y)\}.\]
\end{definition}

\begin{definition}
The {\em quasiconvex envelope} (QCE) $\tilde{g}$ of a function $g$ on a convex subset $\mathcal{D}\in\mathbb{R}$ is defined as
\[\sup\{\tilde{g}(x)|\tilde{g}:\mathbb{R}\rightarrow\mathbb{R}~\mbox{quasiconvex}, \tilde{g}(x)\leq g(x)~\forall x\in\mathcal{D}\}.\]
\end{definition}

\begin{lemma}\label{lemma:nonincreasing_envelope}
The QCE of $G_n$ on interval $(s_m,S_m)$ is nonincreasing.
\end{lemma}
\begin{proof}
From Lemma \ref{lemma:sm_Sm_cost_minimal_maximal}b and Lemma \ref{lemma:sm_Sm_cost_minimal_maximal}c, it follows
$G_n(s_m)> G_n(x)\geq G_n(S_m)$ for $s_m< x< S_m$. Hence, the QCE of $G_n$ on interval $(s_m,S_m)$ is a nonincreasing function.
\end{proof}

\begin{definition}
Consider a function $g:\mathbb{R}\rightarrow\mathbb{R}$, a point $x$ in the domain of $g$ is a {\em strict local minimum from the right} if there exists $\delta>0$ such that $g(y)>g(x)$ for all $y\in(x,x+\delta]$.
\end{definition}

\begin{definition}\label{def:qce_belongs}
Let $[a,b]$, $a\leq b$, in the domain of a function $g$ be a compact interval such that $b$ is a strict local minimum from the right, $g(x)=g(b)$ for all $x\in[a,b]$, and $g(a)=\tilde{g}(a)$; $[a,b]$ {\em nontrivially belongs to the QCE} $\tilde{g}$ of $g$, if there exists $\delta>0$ such that $g(y)>g(x)$ and $g(y)=\tilde{g}(y)$ for all $y\in(a-\delta,a]$; $[a,b]$ {\em trivially belongs to the QCE} $\tilde{g}$ of $g$, if there is no $\delta>0$ such that $g(y)=\tilde{g}(y)$ for all $y\in(a-\delta,a)$.
\end{definition}
The concepts introduced in Definition \ref{def:qce_belongs} are illustrated in Fig. \ref{fig:qce_belongs}.
\begin{figure}
\centering
\includegraphics[scale=0.7]{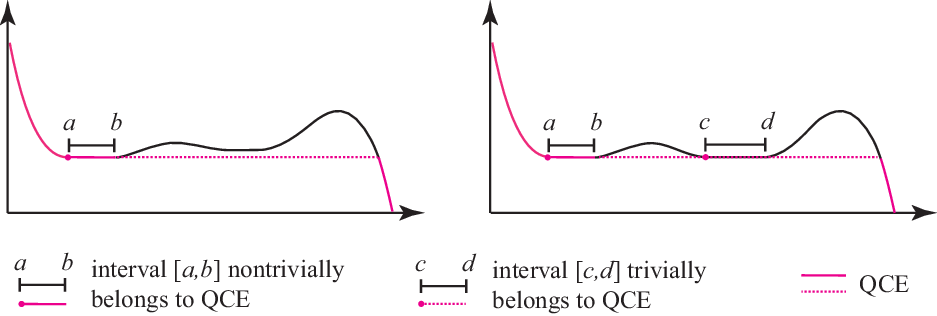}
\caption{Graphical illustration of the concepts introduced in Definition \ref{def:qce_belongs}; note that intervals $[a,b]$ and $[c,d]$ can be degenerate, and reduce to a single point.}\label{fig:qce_belongs}
\end{figure}

Assume $G_n$ is $(K,B)$-convex; this function must be increasing over some intervals in $(s_m,\infty)$, otherwise $G_n(y)$ would be a nonincreasing function for all $y$, which contradicts Lemma \ref{lemma:limiting_behaviour}. 
Let $\widehat{\mathcal{S}}$ be the set of all points $a$ such that interval $[a,b]\in(s_m,S_m)$ {\em nontrivially belongs to the QCE} of $G_n$.

\begin{lemma}\label{lemma:local_minima}
Let $x_0$ be any point at which it is optimal to place an order; then either it is the case that $\arg \min_{y\in(x_0,x_0+B]} G_n(y)=x_0+B$, or that $\arg \min_{y\in(x_0,x_0+B]} G_n(y)=\widehat{S}_k$, for some $\widehat{S}_k\in\widehat{\mathcal{S}}$.
\end{lemma}
\begin{proof}
Assume that at $x_0$ it is optimal to place an order. Then either the lowest cost will be attained by ordering up to $x_0+B$, or by ordering up to some local minimum $S\in (x_0,x_0+B)$. Consider this latter case. We first show that $S$ must belong to the QCE of $G_n$ on $(s_m,S_m)$. Assume, ex absurdo, that $S$ does not belong to the QCE of $G_n$ on $(s_m,S_m)$; since the QCE of $G_n$ is nonincreasing on $(s_m,S_m)$ (Lemma \ref{lemma:nonincreasing_envelope}), there must exist some other local minimum $\widehat{S}$, such that $s_m<\widehat{S}<S$ and $G_n(\widehat{S})<G_n(S)$, which contradicts the fact that at $x_0$ it is optimal to order up to $S$. Finally, assume interval $[S,b]$, for some $b\geq S$, trivially belongs to the QCE of $G_n$ on $(s_m,S_m)$, this means there must exist some other local minimum $\widehat{S}$, such that $s_m<\widehat{S}<S$ and $G_n(\widehat{S})=G_n(S)$; hence ordering up to $S$ is no better than ordering up to $\widehat{S}$. 
\end{proof}
Lemma \ref{lemma:local_minima} is further illustrated in a numerical example presented in Appendix \ref{sec:appendix_5}.

In what follows, we shall assume that $\widehat{\mathcal{S}}\triangleq\{\widehat{S}_1,\widehat{S}_2,\ldots,\widehat{S}_{w-1}\}\subseteq\mathcal{S}$ is an ordered set, so that $s_m<\widehat{S}_1<\widehat{S}_2<\ldots<\widehat{S}_{w-1}<S_m$, and $|\widehat{\mathcal{S}}|\geq 0$.

\begin{lemma}\label{lemma:cost_ordering_of_local_minima}
$G_n(s_m)>G_n(\widehat{S}_1)>G_n(\widehat{S}_2)>\ldots>G_n(\widehat{S}_{w-1})>G_n(S_m)$.
\end{lemma}
\begin{proof}
Immediately follows from the definition of $\widehat{\mathcal{S}}$ and from Lemma \ref{lemma:nonincreasing_envelope}.
\end{proof}

\begin{corollary}
$\widehat{\mathcal{S}}$ is empty if $G_n$ is quasiconvex on $(s_m,S_m)$. 
\end{corollary}
\begin{proof}
If $G_n$ quasiconvex on $(s_m,S_m)$, from Lemma \ref{lemma:nonincreasing_envelope} it follows that $G_n$ is nonincreasing, and hence it does not admit any strict local minima from the right in this interval.
\end{proof}

For the sake of convenience let $\widehat{S}_w\triangleq S_m$.
\begin{lemma}\label{lemma:reorder_points}
For each $\widehat{S}_k\in\widehat{\mathcal{S}}$ there exists a nonempty set $\{b|\widehat{S}_{k}<b<\widehat{S}_{k+1},G_n(b)\geq G_n(\widehat{S}_k)\}$.
\end{lemma}
\begin{proof}
Consider $s_m$ and $S_m$ as defined in Lemma \ref{lemma:s_m}.
From Lemma \ref{lemma:cost_ordering_of_local_minima}, $G_n(\widehat{S}_k)>G_n(\widehat{S}_{k+1})$, for $s_m<\widehat{S}_k<\widehat{S}_{k+1}<S_m$. The result in this lemma follows from the extreme value theorem, since $G_n$ must attain a local maximum at $x^*\in(\widehat{S}_k,\widehat{S}_{k+1})$, such that $G_n(x^*)>G_n(\widehat{S}_k)>G_n(\widehat{S}_{k+1})$.
Note that there cannot be a point $S\in \widehat{\mathcal{S}}$, such that $\widehat{S}_k<S<\widehat{S}_{k+1}$. 
\end{proof}

\begin{definition}\label{def:bk_sk}
For $k=1,\ldots,w-1$, 
$b_k\triangleq \max\{b|\widehat{S}_{k}<b<\widehat{S}_{k+1},G_n(b)\geq G_n(\widehat{S}_k)\}$, and
$s_k\triangleq b_k-B$;
finally, for the sake of convenience, we define $s_0\triangleq -\infty$.
\end{definition}

\begin{lemma}\label{lemma:cool}
$s_{k-1}<\widehat{S}_k-B< s_k$
\end{lemma}
\begin{proof}
This follows from Definition \ref{def:bk_sk}.
\end{proof}

\begin{lemma}\label{lemma:Cn_general_form}
$C_n(x)=-vx+\min\{G_n(x),\min_{x\leq y \leq x+B} G_n(y) + K\}$ takes the general form
\[
C_n(x)=\left\{
\begin{array}{lrlr}
K-vx+G_n(x+B)			&\quad s_{k-1}< x\leq \widehat{S}_{k}-B		&\quad k=1,\ldots,w-1\\[-4pt]
K-vx+G_n(\widehat{S}_k)	&\quad \widehat{S}_{k}-B< x\leq s_{k}     		&\quad k=1,\ldots,w-1\\[-4pt]
K-vx+G_n(x+B)			&\quad s_{w-1}< x\leq S_m-B\\[-4pt]
K-vx+G_n(S_m)		&\quad S_m-B< x\leq s_m	\\[-4pt]
-vx+G_n(x)			&\quad x> s_m.
\end{array}
\right.
\]
\end{lemma}
\begin{proof}
If at $x$ it is optimal to order $a\triangleq S-x$ units, where $a>0$, then $C_n(x)=K-vx+G_n(S)$. We consider each interval for $x$ in order.

$x> s_m$: this case follows from Lemma \ref{lemma:s_m}, since $s_m$ denotes an inventory level beyond which no ordering is optimal. Conversely, because of the continuous order property, for $x\leq s_m$ it is always optimal to order;

$S_m-B< x\leq s_m$: in this interval, $\arg \min_{y\in(x,x+B]} G_n(y)=S_m$, this follows from the definition of $S_m$ (Lemma \ref{lemma:s_m}) and from the fact that $G_n$ is nonincreasing in $(-\infty,s_m]$ (Lemma \ref{lemma:decreasing});

$s_{w-1}< x\leq S_m-B$: in this interval, from Definition \ref{def:bk_sk} it follows that 
$\arg \min_{y\in(x,x+B]} G_n(y)=x+B$, since $G_n(\widehat{S}_k)>G_n(S_m)$, for all $k=1,\ldots,w-1$;

$\widehat{S}_{k}-B< x\leq s_{k}$, for all $k=1,\ldots,w-1$: in this interval, from Definition \ref{def:bk_sk} and from Lemma \ref{lemma:cool}, it follows that $\arg \min_{y\in(x,x+B]} G_n(y)=\widehat{S}_k$;

$s_{k-1}< x\leq \widehat{S}_{k}-B$, for all $k=1,\ldots,w-1$: in this interval, from Definition \ref{def:bk_sk} and from Lemma \ref{lemma:cool}, it follows that $\arg \min_{y\in(x,x+B]} G_n(y)=x+B$, since $G_n(\widehat{S}_k)>G_n(\widehat{S}_{k+1})$;
finally, note that if $s_0<x\leq \widehat{S}_1-B$, then $\arg \min_{y\in(x,x+B]} G_n(y)=x+B$, since $G_n$ is nonincreasing in $(-\infty,\widehat{S}_1]$: in fact, $G_n$ is nonincreasing in $(-\infty,s_m]$ (Lemma \ref{lemma:decreasing}), $\widehat{\mathcal{S}}$ is an ordered set, hence by definition there exists no point $s_m<S<\widehat{S}_1$ that is a strict local minimum from the right, $G_n(s_m)>G_n(\widehat{S}_1)$ (Lemma \ref{lemma:cost_ordering_of_local_minima}), and thus $G_n$ is nonincreasing in $(s_m,\widehat{S}_1]$.
\end{proof}

\begin{definition}\label{def:Sk}
$S_k\triangleq\widehat{S}_k$, for all $k=1,\ldots,w-1$; and, for convenience, let $m\triangleq w$.
\end{definition}

By applying Definition \ref{def:Sk}, we can rewrite, for $k=1,\ldots,m$,
\begin{equation}\label{eq:Cn_policy_skSk}
C_n(x)=\left\{
\begin{array}{lrlr}
K-vx+G_n(x+B)			&\quad s_{k-1}< x\leq S_{k}-B		&\quad \mbox{\em(saturated ordering)}	\\[-4pt]
K-vx+G_n(S_k)			&\quad S_{k}-B< x\leq s_{k}     		&\quad \mbox{\em(unsaturated ordering)}	\\[-4pt]
-vx+G_n(x)			&\quad x> s_m					&\quad \mbox{\em(no order)}
\end{array}
\right.
\end{equation}
where $S_1,\ldots,S_m$ are the order-up-to-levels and $s_1,\ldots,s_m$ the reorder points of the $(s_k, S_k)$ policy.

\begin{corollary}
If the continuous order property holds, the $(s_k, S_k)$ policy generalises the $X-Y$ band discussed in \citep{Chen2004The}.
\end{corollary}
\begin{proof}
In \citeauthor{Chen2004The}, $Y\triangleq s_m$ and $X\triangleq Y-B$, where $X$ denotes an inventory level below which it is optimal to order up to capacity; hence, their $X-Y$ band has size $B$. According to Lemma \ref{lemma:Cn_general_form}, it is optimal to order up to capacity for all $x\leq S_1-B$. According to Lemma \ref{lemma:sm_Sm_cost_minimal_maximal}c, $s_m<S_1$, and thus $s_m-B<S_1-B$. By letting $\bar{X} \triangleq S_1-B$, we obtain a tighter band $\bar{X}-Y$.
\end{proof}

\begin{corollary}
If the continuous order property holds, the $(s_k, S_k)$ policy generalises the policy discussed in \citep{Gallego2000Capacitated}.
\end{corollary}
\begin{proof}
\citeauthor{Gallego2000Capacitated}'s optimal policy structure features two thresholds: $s$ and $s'$, where $-\infty\leq s\leq s'\leq S^*$, and $S^*= \arg \min_y G_n(y)$. 
Clearly, $s'$ is the same threshold we denoted as $s_m$, and under the assumption that the continuous order property holds, it follows that $s=s'$. \citeauthor{Gallego2000Capacitated}'s optimal policy therefore reduces to
\[
C_n(x)=\left\{
\begin{array}{lrr}
K-vx+G_n(x+B)						&\quad x\leq s_m-B			&\quad \mbox{(saturated)}\\
K+\min_{x\leq y\leq x+B} \{G_n(y) -vx\}	&\quad s_m-B< x\leq s_m		&\quad \mbox{(unsaturated or saturated)}\\
-vx+G_n(x)						&\quad x> s_m				&\quad \mbox{(no order)},
\end{array}
\right.
\]
which is equivalent to \citeauthor{Chen2004The}'s X-Y band. 
\end{proof}

\begin{corollary}
If the continuous order property holds, the $(s_k, S_k)$ policy generalises the $(s,S)$ policy discussed in \citep{Scarf1960}.
\end{corollary}
\begin{proof}
When $B=\infty$, $S_m-B=-\infty$, and from Lemma \ref{lemma:Cn_general_form} it is clear that the optimal policy must feature a single reorder threshold $s$ and order-up-to-level $S$. 
\end{proof}

In Fig. \ref{fig:k_convexity_capacity} we illustrate $G_n(y)$ for different ordering capacities ($B\in\{35,65,71,\infty\}$) imposed for the problem in Example \ref{sec:numerical_example}.
\begin{figure}
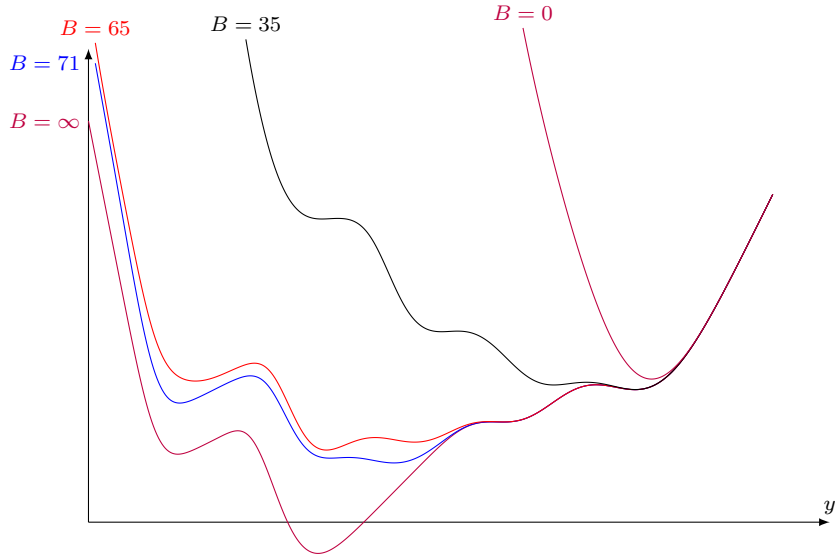

\centering
\include{numerical_example_4}
\caption{Numerical example illustrating $G_n(y)$ for different ordering capacities}
\label{fig:k_convexity_capacity}
\end{figure}
\begin{table}
\centering
\begin{tabular}{rcccccccc}
\toprule
			&\multicolumn{8}{c}{$B$}\\
\cmidrule(l{2pt}r{2pt}){2-9}
			        &\multicolumn{2}{c}{$35$}	&\multicolumn{2}{c}{$65$}	&\multicolumn{2}{c}{$71$}       &\multicolumn{2}{c}{$\infty$}\\
                    \cmidrule(l{2pt}r{2pt}){2-3}\cmidrule(l{2pt}r{2pt}){4-5}\cmidrule(l{2pt}r{2pt}){6-7}\cmidrule(l{2pt}r{2pt}){8-9}
Period			    &$s_k$	&$S_k$			    &$s_k$	&$S_k$			    &$s_k$	&$S_k$                  &$s_k$	&$S_k$\\	
\multirow{3}{*}{1}	&39  		&68				&-11		&31				&-16	&27                     &15	&67\\
				    &46   		&81				&14		    &70				&7		&71                     &	&\\
				    &   		&				&		    &				&13		&84                     &	&\\
\cmidrule(l{2pt}r{2pt}){1-1}	\cmidrule(l{2pt}r{2pt}){2-3}	\cmidrule(l{2pt}r{2pt}){4-5} \cmidrule(l{2pt}r{2pt}){6-7}\cmidrule(l{2pt}r{2pt}){8-9}
\multirow{3}{*}{2}	&64   		&99				&-5		    &51				&27		&76                     &28	&49\\
				    &   		&				&28		    &82				&34		&105                    &	&\\
				    &   		&				&35		    &100			&		&                       &	&\\
\cmidrule(l{2pt}r{2pt}){1-1}	\cmidrule(l{2pt}r{2pt}){2-3}	\cmidrule(l{2pt}r{2pt}){4-5} \cmidrule(l{2pt}r{2pt}){6-7}\cmidrule(l{2pt}r{2pt}){8-9}
\multirow{2}{*}{3}	&61   		&96				&18		    &71				&12		&71                     &55	&109\\
				    &   		&				&55		    &109			&55		&109                    &	&\\
\cmidrule(l{2pt}r{2pt}){1-1}	\cmidrule(l{2pt}r{2pt}){2-3}	\cmidrule(l{2pt}r{2pt}){4-5} \cmidrule(l{2pt}r{2pt}){6-7}\cmidrule(l{2pt}r{2pt}){8-9}
4				    &28   		&49				&28		    &49				&28		&49                     &28	&49\\
\bottomrule
\end{tabular}
\caption{Optimal $(s_k,S_k)$ ordering policy under ordering capacity constraints ($B\in\{35,65,71,\infty\}$) for our numerical example; in all cases the continuous order property holds.}
\label{tab:optimal_sk_Sk_policy}
\end{table}
The optimal $(s_k,S_k)$ ordering policy under ordering capacity constraints for our numerical example is shown in Table \ref{tab:optimal_sk_Sk_policy}, and in Fig. \ref{fig:capacitated_policy_1} for the case in which $B=65$. 

\begin{figure}
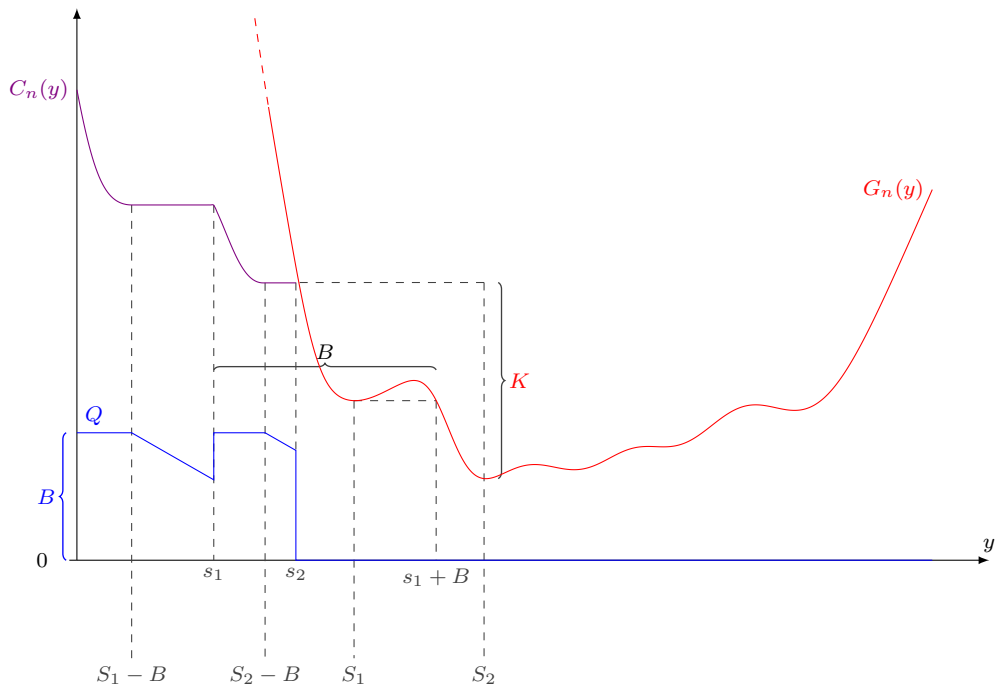

\centering
\include{numerical_example_3}
\caption{Optimal ordering policy in period 1 when $B=65$; note that $G_n(y)$ and $Q$ are not plotted according to the same vertical scale.}
\label{fig:capacitated_policy_1}
\end{figure}

In Appendix \ref{sec:appendix_4} we characterise the structure of the optimal policy for the open numerical example in \citep[][p. 1015]{Chen1996X}, for which the continuous order property holds.

\section{A counterexample}\label{sec:counterexample}

The continuous order property in Definition \ref{def:always_order} has been originally conjectured by \cite{Gallego2000Capacitated}, and it was later further investigated by \cite{Chan2003A}. \cite{Gallego2000Capacitated} wrote:
\begin{quote}
A number of problems still remain. The most vexing is the possibility that under the current structure there could exist a number of intervals [...] where it is optimal to start and stop ordering. An optimal policy with a single continuous interval over which ordering is prescribed, as was found for all of the cases tested [...], is much more analytically appealing. [...]
Unfortunately, the proof of this has thus far eluded us. It should be mentioned that it is likewise possible, although we believe it unlikely, that such a structure simply does not exist. To show this requires a problem instance in which the optimal policy has multiple disjoint intervals in which ordering is optimal. Our computational study suggests that this is not the case.
\end{quote}
\cite{Chan2003A} wrote:
\begin{quote}
If our conjecture [the continuous order property] holds, the computational time for obtaining the optimal ordering policy parameters can be further reduced [...]. We can only show that this conjecture holds for a special case where [the capacity] is large enough [...]. It should be an interesting problem for researchers to prove or disprove the conjecture is true for small [capacity].
\end{quote}
In the rest of this section, we introduce a numerical instance that violates the continuous order property. To the best of our knowledge, no such instance has ever been discussed in the literature.

\begin{example}\label{example:counterexample}
Consider a planning horizon of $n=4$ periods and a nonstationary demand $d_t$ distributed in each period $t$ according to the probability mass function shown in Table \ref{tab:counterexample}. Other problem parameters are $K=250$, $B=41$, $h=1$ and $p=26$ and $v=0$.
\end{example}

\begin{table}
\centering
\begin{tabular}{lrrrrrrrrrrr}
\toprule
$d_1$	&34 (0.018)	&159 (0.888)	&281 (0.046)	&286 (0.048)\\
$d_2$	&14 (0.028)	&223 (0.271)	&225 (0.170)	&232 (0.531)\\
$d_3$	&5 (0.041)		&64 (0.027)	&115 (0.889)	&171 (0.043)	\\
$d_4$	&35 (0.069)	&48 (0.008)	&145 (0.019)	&210 (0.904)	\\
\bottomrule
\end{tabular}
\caption{Probability mass functions of the nonstationary demand $d_t$ considered in Example \ref{example:counterexample}.}
\label{tab:counterexample}
\end{table}

In Table \ref{tab:counterexample_opt} we report an extract of the tabulated optimal policy in which the continuous order property is violated (Fig. \ref{fig:CnMinusGn_violated_9c}).

\begin{table}[ht!]
\centering
\begin{tabular}{lrrrrrrrrrr}
\toprule
Starting inventory level	&593&594&595&596&597&598&599&600&601\\
Optimal order quantity	&41	&40	&39	&38	&37	&36	&35	&34	&33\\
\midrule
Starting inventory level	&602&603&604&605&606&607&608&609&610\\
Optimal order quantity	&0	&0	&0	&0	&0	&0	&0	&0	&0	\\
\midrule
Starting inventory level	&611&612&613&614&615&616&617&618&619\\
Optimal order quantity	&0	&0	&0	&0	&0	&41	&41	&41	&0	\\
\bottomrule
\end{tabular}
\caption{An extract of the optimal policy for period $t=1$ of Example \ref{example:counterexample} in which the continuous order property is violated.}
\label{tab:counterexample_opt}
\end{table}

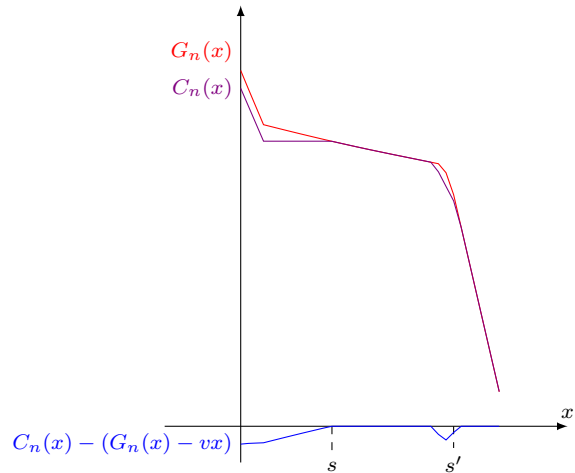
\begin{figure}
\centering
\begin{tikzpicture}[x=0.1cm, y=0.024cm]
\footnotesize
\draw [-latex] ([xshift=-0mm] 20,0) -- ([xshift=3mm] 70.0,0) node[above] {$x$};
\draw [-latex] ([yshift=-0mm] 30,-20.0) -- ([yshift=3mm] 30, 220.0);

\draw [color=red, mark= , style=solid] plot coordinates {%
(590-560,1896.5770144277064-1700)
(591-560,1886.451328561398-1700)
(592-560,1876.533038984386-1700)
(593-560,1866.6158997800662-1700)
(594-560,1865.558889256914-1700)
(595-560,1864.503828851874-1700)
(596-560,1863.458027640642-1700)
(597-560,1862.4261471410919-1700)
(598-560,1861.3992895873341-1700)
(599-560,1860.395039394536-1700)
(600-560,1859.39556034604-1700)
(601-560,1858.405670841752-1700)
(602-560,1857.43030729748-1700)
(603-560,1856.4833749713682-1700)
(604-560,1855.5442740257363-1700)
(605-560,1854.605526114392-1700)
(606-560,1853.6903413374798-1700)
(607-560,1852.7847771081201-1700)
(608-560,1851.90888023684-1700)
(609-560,1851.0329833655603-1700)
(610-560,1850.1573133629681-1700)
(611-560,1849.30610490764-1700)
(612-560,1848.454896452312-1700)
(613-560,1847.6036879969843-1700)
(614-560,1846.7524795416557-1700)
(615-560,1845.901507818872-1700)
(616-560,1845.0505360960879-1700)
(617-560,1840.0721441366481-1700)
(618-560,1827.956954866016-1700)
(619-560,1809.813998519816-1700)
(620-560,1791.6710421736163-1700)
(621-560,1773.5280858274164-1700)
(622-560,1755.3851294812162-1700)
(623-560,1737.242173135016-1700)
(624-560,1719.099216788816-1700)
};

 \draw[color=red] (30, 1896.5770144277064-1700) node[above left] {$G_n(x)$};

  \draw [color=violet, mark= , style=solid] plot coordinates 
 {%
(590-560,1886.7126431630963-1700)
(591-560,1876.977416991096-1700)
(592-560,1867.238197259032-1700)
(593-560,1857.4936513317684-1700)
(594-560,1857.4936513317684-1700)
(595-560,1857.4936513317684-1700)
(596-560,1857.4936513317684-1700)
(597-560,1857.4936513317684-1700)
(598-560,1857.4936513317684-1700)
(599-560,1857.4936513317684-1700)
(600-560,1857.4936513317684-1700)
(601-560,1857.4936513317684-1700)
(602-560,1857.43030729748-1700)
(603-560,1856.4833749713682-1700)
(604-560,1855.5442740257363-1700)
(605-560,1854.605526114392-1700)
(606-560,1853.6903413374798-1700)
(607-560,1852.7847771081201-1700)
(608-560,1851.90888023684-1700)
(609-560,1851.0329833655603-1700)
(610-560,1850.1573133629681-1700)
(611-560,1849.30610490764-1700)
(612-560,1848.454896452312-1700)
(613-560,1847.6036879969843-1700)
(614-560,1846.7524795416557-1700)
(615-560,1845.901507818872-1700)
(616-560,1840.5854334929018-1700)
(617-560,1832.5032933370223-1700)
(618-560,1824.520638236502-1700)
(619-560,1809.813998519816-1700)
(620-560,1791.6710421736163-1700)
(621-560,1773.5280858274164-1700)
(622-560,1755.3851294812162-1700)
(623-560,1737.242173135016-1700)
(624-560,1719.099216788816-1700)
 };
 
 \draw[color=violet] (30, 1886.7126431630963-1700) node[left] {$C_n(x)$};

\draw [color=black, mark= , style=dashed] (618-560,0) -- (618-560,-12) node[below] {$s'$};
\draw [color=black, mark= , style=dashed] (602-560,0) -- (602-560,-15) node[below] {$s$};

\draw [color=blue, mark= , style=solid,] plot coordinates {
(590-560,-9.8643712646101)
(591-560,-9.47391157030188)
(592-560,-9.29484172535399)
(593-560,-9.122248448297796)
(594-560,-8.065237925145539)
(595-560,-7.010177520105572)
(596-560,-5.964376308873625)
(597-560,-4.932495809323427)
(598-560,-3.905638255565691)
(599-560,-2.901388062767637)
(600-560,-1.9019090142714958)
(601-560,-0.9120195099835655)
(602-560,0.0)
(603-560,0.0)
(604-560,0.0)
(605-560,0.0)
(606-560,0.0)
(607-560,0.0)
(608-560,0.0)
(609-560,0.0)
(610-560,0.0)
(611-560,0.0)
(612-560,0.0)
(613-560,0.0)
(614-560,0.0)
(615-560,0.0)
(616-560,-4.4651026031860965)
(617-560,-7.568850799625807)
(618-560,-3.4363166295140672)
(619-560,0.0)
(620-560,0.0)
(621-560,0.0)
(622-560,0.0)
(623-560,0.0)
(624-560,0.0)
};

\draw[color=blue] (30, -9.8643712646101) node[left] {$C_n(x)-(G_n(x)-vx)$};
  
\end{tikzpicture}
\caption{Lemma \ref{lemma:cop_convex_set} does not hold in the context of Example \ref{example:counterexample}: $\{x|C_n(x)-(G_n(x)-vx)<0\}$ is not a convex set; hence the continuous order property is violated and $s<s'$.}
\label{fig:CnMinusGn_violated_9c}
\end{figure}

Our numerical example confirms that it is possible to construct instances for which it is optimal to start and stop ordering, and that the continuous order property conjectured in \citep{Gallego2000Capacitated,Chan2003A} does not hold for the general case of the stochastic inventory problem under order quantity capacity constraints. In Appendix \ref{sec:appendix_6} we discuss the rationale underpinning the generation of our counterexample.

\section{Computational study}\label{sec:computational_study}

Albeit in the previous section we demonstrated that the it is possible to construct instances for which the continuous order property does not hold, we must underscore that these instances are incredibly rare. This is also the reason why the conjecture in \citep{Gallego2000Capacitated,Chan2003A} remained open for over twenty years. In this section, we consider an extensive test bed comprising a broad family of demand distributions and problem parameters; our aim is threefold.

First, we aim to show empirically that, in practice, instances that violate the continuous order property are extremely rare. In turn, this means that the plans generated by the modified multi-$(s,S)$ ordering policy can be considered, for all practical purposes, optimal.

Second, the modified multi-$(s,S)$ ordering policy may feature, in each period, a variable number of thresholds $s_k$ and associated order-up-to-levels $S_k$. In our computational study, we show that the number of thresholds in a modified multi-($s,S$) policy is typically very low in each period. This means that operating this policy is generally not too cumbersome, as the decision maker only needs to track a few (usually less than 5) reorder thresholds and associated order up to levels.

Finally, as shown in Table \ref{tab:optimal_sk_Sk_policy_heuristic}, a modified ($s,S$) policy \citep{Wijngaard1972} with parameters $(s_m,S_m)$ appears to perform well in the context of Example \ref{sec:numerical_example}; in our study we proceed to show that this simple policy, which has been known for decades, also performs well in practice across all instances considered.  

\begin{table}
\centering
\begin{tabular}{rcccccc}
\toprule
            &\multicolumn{6}{c}{$B$}\\
\cmidrule(l{2pt}r{2pt}){2-7}
			&\multicolumn{2}{c}{$35$}	&\multicolumn{2}{c}{$65$}	&\multicolumn{2}{c}{$71$}\\
	        \cmidrule(l{2pt}r{2pt}){2-3}\cmidrule(l{2pt}r{2pt}){4-5}\cmidrule(l{2pt}r{2pt}){6-7}
Period			&$s_m$	&$S_m$			&$s_m$	&$S_m$			&$s_m$	&$S_m$\\
1				&46		&81				&14		&70				&13		&84\\		
2				&64		&99				&35		&100				&34		&105\\			
3				&61		&96				&55		&109				&55		&109\\
4				&28		&49				&28		&49				&28		&49\\		
\cmidrule(l{2pt}r{2pt}){1-1}	\cmidrule(l{2pt}r{2pt}){2-3}	\cmidrule(l{2pt}r{2pt}){4-5}	\cmidrule(l{2pt}r{2pt}){6-7}		
Optimality gap (\%)	&\multicolumn{2}{c}{0.000}	&\multicolumn{2}{c}{0.123}	&\multicolumn{2}{c}{0.192}\\
\bottomrule
\end{tabular}
\caption{Modified $(s,S)$ ordering policies for Example \ref{sec:numerical_example} when $B\in\{35,65,71\}$.}
\label{tab:optimal_sk_Sk_policy_heuristic}
\end{table}

\subsection{Test bed}\label{sec:test_bed}

In our test bed, the planning horizon comprises $n=20$ periods. We consider 10 different patterns for the expected value of the demand in each period, as shown in Fig. \ref{fig:demand}: 2 life cycle patterns (LCY1 and LCY2), 2 sinusoidal patterns (SIN1 and SIN2), 1 stationary pattern (STA), 1 random pattern (RAND), and 4 empirical patterns (EMP1, EMP2, EMP3, EMP4) derived from demand data in \citep{Kurawarwala_1996}. Further details of expected demand rates in each period are given in Table \ref{table:demand_data} in  Appendix \ref{sec:appendix_3}. 
\begin{figure}
\centering
\includegraphics[scale=0.6]{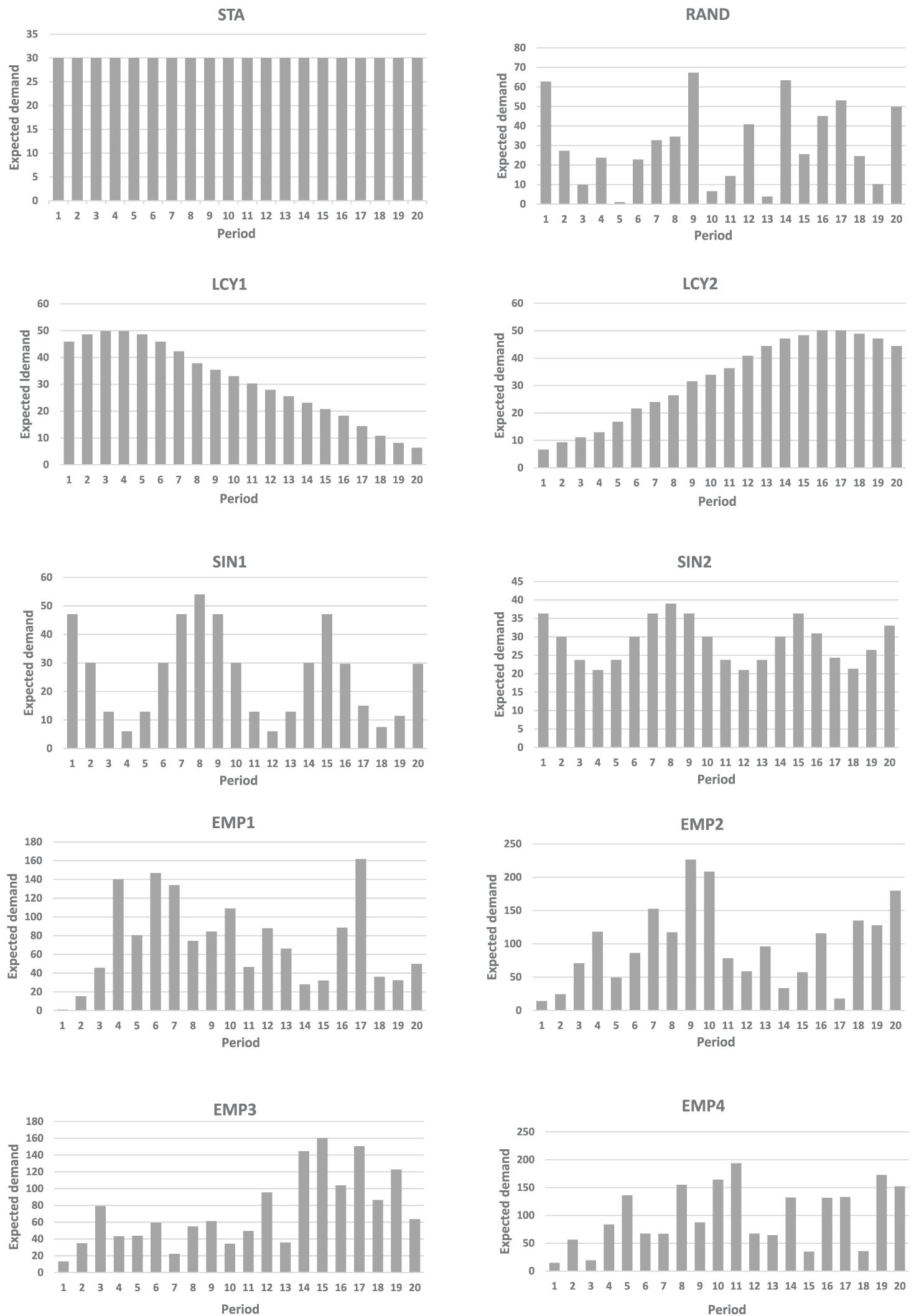}
\caption{Demand patterns in our computational study}\label{fig:demand}
\end{figure}

We consider a broad family of demand distributions commonly used in practice: discrete uniform, geometric, Poisson, normal, lognormal, and gamma. Demands in different periods are assumed to be mutually independent. More specifically, let $\mu_t$ denote the mean demand in period $t$, we investigate a demand that follows a discrete uniform distribution in $[0,2\mu_t)$; a demand that follows a geometric distribution with mean $\mu_t$; and a demand that follows a Poisson distribution with rate $\mu_t$. Finally, given the coefficient of variation of the demand in each period $c_v=\sigma_t/\mu_t$, where $\sigma_t$ is the standard deviation of the demand in period $t$; we consider a demand that follows a normal, a lognormal, and a gamma distribution with mean $\mu_t$ and standard deviation $\sigma_t$.

Fixed ordering cost $K$ takes values in $\{250, 500, 1000\}$; inventory holding cost $h$ is 1; unit variable ordering cost $v$ takes values in $\{2, 5, 10\}$; unit penalty cost $p$ ranges in $\{5, 10, 15\}$. For the case of normal, lognormal, and gamma distributed demand, the coefficient of variation takes values in $\{0.1,0.2,0.3\}$.
Let $D$ denote the average demand rate over the whole $n$ periods horizon for a given demand pattern; the maximum order quantity $B$ takes values in $\{\text{round}(2D)$, $\text{round}(3D)$, $\text{round}(4D)\}$, where the $\text{round}$ operator rounds the value to the nearest integer.  

Since we adopt a full factorial design, we consider 810 instances for discrete uniform, geometric, and Poisson distributed demand, respectively; and 2430 instances for normal, lognormal, and gamma distributed demands, respectively, since in these latter cases we must also consider the three levels of the coefficient of variation. In total, our computational study comprises 9720 instances.

\subsection{Results}\label{sec:results}

We run experiments on an Intel(R) Xeon(R) @ 3.5GHz with 16Gb of RAM.
\footnote{
The library used in our experiments is \texttt{jsdp}.
The Java code is available on \url{http://gwr3n.github.io/jsdp/}.
A self-contained Python code is also available on \url{https://github.com/gwr3n/inventoryanalytics}.}
SDP state space boundaries are fixed --- inventory may range in $(-10000, 10000)$ --- and in all cases we adopt a unit discretization, therefore running time for each instance is constant; a continuity correction is introduced for continuous distributions. Monte Carlo simulation runs are determined by targeting an estimation error of 0.01\% for the mean estimated at 95\% confidence level; we adopt a common random number strategy \citep{Kahn1953} across all instances.

\begin{table}
\scriptsize
\centering
\begin{tabular}{llrrcr}
\toprule
\multicolumn{2}{c}{}&\multicolumn{2}{c}{modified $(s, S)$} & \multicolumn{1}{c}{modified multi-$(s, S)$}\\
\multicolumn{2}{c}{}& \multicolumn{2}{c}{\% optimality gap} & \multicolumn{1}{c}{max thresholds} & \multicolumn{1}{c}{instances}	\\
\cmidrule(l{2pt}r{2pt}){3-4} \cmidrule(l{2pt}r{2pt}){5-5} \cmidrule(l{2pt}r{2pt}){6-6} \multicolumn{2}{c}{}
				&avg			&max		\\
\multirow{3}{*}{$K$}										
&	250	&	0.004	&	0.122	&	3	&	270	\\
&	500	&	0.000	&	0.050	&	3	&	270	\\
&	1000	&	0.000	&	0.007	&	3	&	270	\\
\hline										
\multirow{3}{*}{$v$}										
&	2	&	0.002	&	0.122	&	3	&	270	\\
&	5	&	0.000	&	0.057	&	3	&	270	\\
&	10	&	0.000	&	0.032	&	3	&	270	\\
\hline										
\multirow{3}{*}{$p$}										
&	5	&	0.000	&	0.057	&	3	&	270	\\
&	10	&	0.001	&	0.122	&	3	&	270	\\
&	15	&	0.001	&	0.115	&	3	&	270	\\
\hline										
\multirow{3}{*}{$B$}										
&	2.0D	&	0.000	&	0.047	&	2	&	270	\\
&	3.0D	&	0.001	&	0.122	&	3	&	270	\\
&	4.0D	&	0.000	&	0.062	&	3	&	270	\\
\hline										
\multirow{10}{*}{Demand}										
&	EMP1	&	0.002	&	0.122	&	3	&	81	\\
&	EMP2	&	0.003	&	0.044	&	3	&	81	\\
&	EMP3	&	0.003	&	0.039	&	3	&	81	\\
&	EMP4	&	0.003	&	0.057	&	3	&	81	\\
&	LC1	&	0.000	&	0.000	&	1	&	81	\\
&	LC2	&	0.002	&	0.008	&	1	&	81	\\
&	RAND	&	0.006	&	0.018	&	2	&	81	\\
&	SIN1	&	0.000	&	0.002	&	3	&	81	\\
&	SIN2	&	0.000	&	0.001	&	1	&	81	\\
&	STA	&	0.000	&	0.001	&	1	&	81	\\
\hline										
Overall 	&	&	0.000	&	0.122	&	3	&	810	\\
\bottomrule
\end{tabular}
\caption{Pivot table for our computational study: discrete uniform demand.}
\label{table:results_discrete_uniform}
\end{table}

\begin{table}
\scriptsize
\centering
\begin{tabular}{llrrcr}
\toprule
\multicolumn{2}{c}{}&\multicolumn{2}{c}{modified $(s, S)$} & \multicolumn{1}{c}{modified multi-$(s, S)$}\\
\multicolumn{2}{c}{}& \multicolumn{2}{c}{\% optimality gap} & \multicolumn{1}{c}{max thresholds} & \multicolumn{1}{c}{instances}	\\
\cmidrule(l{2pt}r{2pt}){3-4} \cmidrule(l{2pt}r{2pt}){5-5} \cmidrule(l{2pt}r{2pt}){6-6} \multicolumn{2}{c}{}
				&avg			&max		\\
\multirow{3}{*}{$K$}										
&	250	&	0.274	&	0.610	&	3	&	270	\\
&	500	&	0.238	&	0.530	&	3	&	270	\\
&	1000	&	0.200	&	0.427	&	2	&	270	\\
\hline										
\multirow{3}{*}{$v$}										
&	2	&	0.284	&	0.610	&	3	&	270	\\
&	5	&	0.235	&	0.478	&	3	&	270	\\
&	10	&	0.193	&	0.383	&	3	&	270	\\
\hline										
\multirow{3}{*}{$p$}										
&	5	&	0.174	&	0.368	&	2	&	270	\\
&	10	&	0.242	&	0.510	&	3	&	270	\\
&	15	&	0.296	&	0.610	&	2	&	270	\\
\hline										
\multirow{3}{*}{$B$}										
&	2.0D	&	0.279	&	0.610	&	2	&	270	\\
&	3.0D	&	0.234	&	0.495	&	2	&	270	\\
&	4.0D	&	0.199	&	0.416	&	3	&	270	\\
\hline										
\multirow{10}{*}{Demand}										
&	EMP1	&	0.279	&	0.596	&	2	&	81	\\
&	EMP2	&	0.280	&	0.602	&	3	&	81	\\
&	EMP3	&	0.240	&	0.488	&	2	&	81	\\
&	EMP4	&	0.272	&	0.610	&	2	&	81	\\
&	LC1	&	0.241	&	0.573	&	1	&	81	\\
&	LC2	&	0.205	&	0.435	&	1	&	81	\\
&	RAND	&	0.239	&	0.501	&	2	&	81	\\
&	SIN1	&	0.224	&	0.490	&	2	&	81	\\
&	SIN2	&	0.199	&	0.451	&	1	&	81	\\
&	STA	&	0.196	&	0.452	&	1	&	81	\\
\hline										
Overall 	&	&	0.237	&	0.610	&	3	&	810	\\
\bottomrule
\end{tabular}
\caption{Pivot table for our computational study: geometric demand.}
\label{table:results_geometric}
\end{table}

\begin{table}
\scriptsize
\centering
\begin{tabular}{llrrcr}
\toprule
\multicolumn{2}{c}{}&\multicolumn{2}{c}{modified $(s, S)$} & \multicolumn{1}{c}{modified multi-$(s, S)$}\\
\multicolumn{2}{c}{}& \multicolumn{2}{c}{\% optimality gap} & \multicolumn{1}{c}{max thresholds} & \multicolumn{1}{c}{instances}	\\
\cmidrule(l{2pt}r{2pt}){3-4} \cmidrule(l{2pt}r{2pt}){5-5} \cmidrule(l{2pt}r{2pt}){6-6} \multicolumn{2}{c}{}
				&avg			&max		\\
\multirow{3}{*}{$K$}										
&	250	&	0.125	&	1.918	&	5	&	270	\\
&	500	&	0.130	&	1.583	&	5	&	270	\\
&	1000	&	0.029	&	0.424	&	5	&	270	\\
\hline										
\multirow{3}{*}{$v$}										
&	2	&	0.146	&	1.918	&	5	&	270	\\
&	5	&	0.086	&	0.972	&	5	&	270	\\
&	10	&	0.052	&	0.650	&	5	&	270	\\
\hline										
\multirow{3}{*}{$p$}										
&	5	&	0.070	&	1.048	&	4	&	270	\\
&	10	&	0.100	&	1.623	&	5	&	270	\\
&	15	&	0.114	&	1.918	&	5	&	270	\\
\hline										
\multirow{3}{*}{$B$}										
&	2.0D	&	0.103	&	1.918	&	4	&	270	\\
&	3.0D	&	0.100	&	1.623	&	4	&	270	\\
&	4.0D	&	0.081	&	1.583	&	5	&	270	\\
\hline										
\multirow{10}{*}{Demand}										
&	EMP1	&	0.204	&	1.623	&	4	&	81	\\
&	EMP2	&	0.176	&	1.918	&	4	&	81	\\
&	EMP3	&	0.181	&	1.479	&	5	&	81	\\
&	EMP4	&	0.248	&	1.583	&	5	&	81	\\
&	LC1	&	0.018	&	0.154	&	4	&	81	\\
&	LC2	&	0.027	&	0.160	&	4	&	81	\\
&	RAND	&	0.043	&	0.429	&	4	&	81	\\
&	SIN1	&	0.016	&	0.088	&	3	&	81	\\
&	SIN2	&	0.017	&	0.106	&	4	&	81	\\
&	STA	&	0.016	&	0.093	&	5	&	81	\\
\hline										
Overall 	&	&	0.095	&	1.918	&	5	&	810	\\
\bottomrule
\end{tabular}
\caption{Pivot table for our computational study: Poisson demand.}
\label{table:results_poisson}
\end{table}

\begin{table}
\scriptsize
\centering
\begin{tabular}{llrrcr}
\toprule
\multicolumn{2}{c}{}&\multicolumn{2}{c}{modified $(s, S)$} & \multicolumn{1}{c}{modified multi-$(s, S)$}\\
\multicolumn{2}{c}{}& \multicolumn{2}{c}{\% optimality gap} & \multicolumn{1}{c}{max thresholds} & \multicolumn{1}{c}{instances}	\\
\cmidrule(l{2pt}r{2pt}){3-4} \cmidrule(l{2pt}r{2pt}){5-5} \cmidrule(l{2pt}r{2pt}){6-6} \multicolumn{2}{c}{}
				&avg			&max		\\
\multirow{3}{*}{$K$}										
&	250	&	0.092	&	2.006	&	5	&	810	\\
&	500	&	0.056	&	1.435	&	6	&	810	\\
&	1000	&	0.015	&	0.565	&	6	&	810	\\
\hline										
\multirow{3}{*}{$v$}										
&	2	&	0.083	&	2.006	&	6	&	810	\\
&	5	&	0.050	&	0.971	&	6	&	810	\\
&	10	&	0.029	&	0.476	&	6	&	810	\\
\hline										
\multirow{3}{*}{$p$}										
&	5	&	0.034	&	0.893	&	5	&	810	\\
&	10	&	0.060	&	1.597	&	6	&	810	\\
&	15	&	0.068	&	2.006	&	6	&	810	\\
\hline										
\multirow{3}{*}{$B$}										
&	2.0D	&	0.050	&	2.006	&	5	&	810	\\
&	3.0D	&	0.068	&	1.597	&	5	&	810	\\
&	4.0D	&	0.045	&	1.435	&	6	&	810	\\
\hline										
\multirow{10}{*}{Demand}										
&	EMP1	&	0.104	&	1.597	&	4	&	243	\\
&	EMP2	&	0.088	&	2.006	&	4	&	243	\\
&	EMP3	&	0.092	&	1.435	&	6	&	243	\\
&	EMP4	&	0.120	&	1.392	&	5	&	243	\\
&	LC1	&	0.016	&	0.357	&	5	&	243	\\
&	LC2	&	0.019	&	0.910	&	6	&	243	\\
&	RAND	&	0.040	&	1.347	&	5	&	243	\\
&	SIN1	&	0.024	&	0.437	&	5	&	243	\\
&	SIN2	&	0.024	&	0.742	&	5	&	243	\\
&	STA	&	0.015	&	0.327	&	5	&	243	\\
\hline										
\multirow{3}{*}{$c_v$}										
&	0.1	&	0.111	&	2.006	&	6	&	810	\\
&	0.2	&	0.047	&	1.237	&	5	&	810	\\
&	0.3	&	0.006	&	0.565	&	4	&	810	\\
\hline										
Overall 	&	&	0.054	&	2.006	&	6	&	2430	\\
\bottomrule
\end{tabular}
\caption{Pivot table for our computational study: normal demand.}
\label{table:results_normal}
\end{table}

\begin{table}
\scriptsize
\centering
\begin{tabular}{llrrcr}
\toprule
\multicolumn{2}{c}{}&\multicolumn{2}{c}{modified $(s, S)$} & \multicolumn{1}{c}{modified multi-$(s, S)$}\\
\multicolumn{2}{c}{}& \multicolumn{2}{c}{\% optimality gap} & \multicolumn{1}{c}{max thresholds} & \multicolumn{1}{c}{instances}	\\
\cmidrule(l{2pt}r{2pt}){3-4} \cmidrule(l{2pt}r{2pt}){5-5} \cmidrule(l{2pt}r{2pt}){6-6} \multicolumn{2}{c}{}
				&avg			&max		\\
\multirow{3}{*}{$K$}										
&	250	&	0.108	&	1.891	&	5	&	810	\\
&	500	&	0.072	&	1.424	&	6	&	810	\\
&	1000	&	0.033	&	0.579	&	6	&	810	\\
\hline										
\multirow{3}{*}{$v$}										
&	2	&	0.099	&	1.891	&	6	&	810	\\
&	5	&	0.067	&	0.931	&	6	&	810	\\
&	10	&	0.047	&	0.729	&	6	&	810	\\
\hline										
\multirow{3}{*}{$p$}										
&	5	&	0.050	&	0.893	&	5	&	810	\\
&	10	&	0.076	&	1.578	&	6	&	810	\\
&	15	&	0.086	&	1.891	&	6	&	810	\\
\hline										
\multirow{3}{*}{$B$}										
&	2.0D	&	0.066	&	1.891	&	5	&	810	\\
&	3.0D	&	0.086	&	1.578	&	5	&	810	\\
&	4.0D	&	0.061	&	1.424	&	6	&	810	\\
\hline										
\multirow{10}{*}{Demand}										
&	EMP1	&	0.130	&	1.578	&	4	&	243	\\
&	EMP2	&	0.100	&	1.891	&	4	&	243	\\
&	EMP3	&	0.103	&	1.424	&	6	&	243	\\
&	EMP4	&	0.139	&	1.285	&	5	&	243	\\
&	LC1	&	0.033	&	0.286	&	5	&	243	\\
&	LC2	&	0.035	&	0.887	&	6	&	243	\\
&	RAND	&	0.056	&	1.262	&	5	&	243	\\
&	SIN1	&	0.042	&	0.600	&	5	&	243	\\
&	SIN2	&	0.041	&	0.695	&	5	&	243	\\
&	STA	&	0.030	&	0.311	&	5	&	243	\\
\hline										
\multirow{3}{*}{$c_v$}										
&	0.1	&	0.110	&	1.891	&	6	&	810	\\
&	0.2	&	0.055	&	0.923	&	5	&	810	\\
&	0.3	&	0.048	&	0.579	&	4	&	810	\\
\hline										
Overall 	&	&	0.071	&	1.891	&	6	&	2430	\\
\bottomrule
\end{tabular}
\caption{Pivot table for our computational study: lognormal demand.}
\label{table:results_lognormal}
\end{table}

\begin{table}
\scriptsize
\centering
\begin{tabular}{llrrcr}
\toprule
\multicolumn{2}{c}{}&\multicolumn{2}{c}{modified $(s, S)$} & \multicolumn{1}{c}{modified multi-$(s, S)$}\\
\multicolumn{2}{c}{}& \multicolumn{2}{c}{\% optimality gap} & \multicolumn{1}{c}{max thresholds} & \multicolumn{1}{c}{instances}	\\
\cmidrule(l{2pt}r{2pt}){3-4} \cmidrule(l{2pt}r{2pt}){5-5} \cmidrule(l{2pt}r{2pt}){6-6} \multicolumn{2}{c}{}
				&avg			&max		\\
\multirow{3}{*}{$K$}										
&	250	&	0.101	&	1.930	&	5	&	810	\\
&	500	&	0.068	&	1.424	&	6	&	810	\\
&	1000	&	0.029	&	0.570	&	6	&	810	\\
\hline										
\multirow{3}{*}{$v$}										
&	2	&	0.094	&	1.930	&	6	&	810	\\
&	5	&	0.062	&	0.923	&	6	&	810	\\
&	10	&	0.043	&	0.731	&	6	&	810	\\
\hline										
\multirow{3}{*}{$p$}										
&	5	&	0.046	&	0.894	&	5	&	810	\\
&	10	&	0.071	&	1.585	&	6	&	810	\\
&	15	&	0.081	&	1.930	&	6	&	810	\\
\hline										
\multirow{3}{*}{$B$}										
&	2.0D	&	0.062	&	1.930	&	5	&	810	\\
&	3.0D	&	0.080	&	1.585	&	5	&	810	\\
&	4.0D	&	0.057	&	1.424	&	6	&	810	\\
\hline										
\multirow{10}{*}{Demand}										
&	EMP1	&	0.125	&	1.585	&	4	&	243	\\
&	EMP2	&	0.095	&	1.930	&	4	&	243	\\
&	EMP3	&	0.100	&	1.424	&	6	&	243	\\
&	EMP4	&	0.130	&	1.312	&	5	&	243	\\
&	LC1	&	0.029	&	0.308	&	5	&	243	\\
&	LC2	&	0.032	&	0.894	&	6	&	243	\\
&	RAND	&	0.052	&	1.290	&	5	&	243	\\
&	SIN1	&	0.036	&	0.421	&	5	&	243	\\
&	SIN2	&	0.037	&	0.708	&	5	&	243	\\
&	STA	&	0.027	&	0.317	&	5	&	243	\\
\hline										
\multirow{3}{*}{$c_v$}										
&	0.1	&	0.109	&	1.930	&	6	&	810	\\
&	0.2	&	0.050	&	1.013	&	5	&	810	\\
&	0.3	&	0.040	&	0.570	&	4	&	810	\\
\hline										
Overall 	&	&	0.066	&	1.930	&	6	&	2430	\\
\bottomrule
\end{tabular}
\caption{Pivot table for our computational study: gamma demand.}
\label{table:results_gamma}
\end{table}

In Tables \ref{table:results_discrete_uniform}--\ref{table:results_gamma} we present the results of our study for each of the demand distributions under scrutiny.
For all instances investigated, a modified multi-$(s, S)$ policy is optimal. Moreover, the maximum number of thresholds observed in any given period is surprisingly low and never exceeds 6 over the whole test bed. We also report the average and maximum \% optimality gap of a modified $(s, S)$ policy with parameters $(s_m,S_m)$ extracted from the SDP tables. This policy is found to be near optimal in our study, since its average \% optimality gap is consistently negligible, while the maximum \% optimality gap observed never exceeds 2\%.

\section{Conclusions}\label{sec:conclusions}

The periodic review single-item single-stocking location stochastic inventory system under nonstationary demand, complete backorders, a fixed ordering cost component, and order quantity capacity constraints is one of the fundamental problems in inventory management.

A long standing open question in the literature is whether a policy with a single continuous interval over which ordering is prescribed is optimal for this problem. The so-called ``continuous order property'' conjecture was originally posited by \cite{Gallego2000Capacitated}, and later also investigated by \cite{Chan2003A}. To the best of our knowledge, to date this conjecture has never been confirmed or disproved.

In this work, we provided a numerical counterexample that violates the continuous order property. This closes a fundamental and long standing problem in the literature: a policy with a single continuous interval over which ordering is prescribed is not optimal.

\cite{Gallego2000Capacitated} provided a partial characterisation of the optimal policy to the problem. In light of the results presented in \citep{Chen2004The}, we showed how to simplify the optimal policy structure presented by \cite{Gallego2000Capacitated}. \cite{Gallego2000Capacitated} also briefly sketched the form that an optimal policy would take under moderate values of $K$. We formalised this discussion and provided a full characterisation of the optimal policy for instances for which the continuous order property holds. In particular, we showed that under this assumption the optimal policy takes the {\em modified multi-}$(s,S)$ {\em form}. 

By leveraging an extensive computational study, we showed that instances violating the continuous order property are extremely rare in practice. The modified multi-$(s,S)$ ordering policy can therefore be considered, for all practical purposes, optimal. Moreover, we observed that the number of thresholds in a modified multi-($s,S$) policy remains low in each period, this means that operating the policy in practice will not result too cumbersome for a manager. Finally, we showed that a well-known heuristic policy which has been known for decades, the modified ($s,S$) policy \citep{Wijngaard1972}, also performs well in practice across all instances considered.  

Since a policy with a single continuous interval over which ordering is prescribed is not optimal in general, future works may focus on establishing what restrictions (if any) to the problem statement, e.g. nature of the demand distribution, may ensure that a policy with a single continuous interval over which ordering is prescribed is optimal.

\section{Acknowledgments}

The authors would like to thank the China Scholarship Council (CSC) for the financial support provided to Z. Chen under the CSC Postgraduate Study Abroad Program.

\begin{appendices}

\section{Possible scenarios one may observe when inventory hits level $s_m$}\label{sec:appendix_1}

There are two possible cases one may encounter when inventory hits reorder threshold $s_m$: either we order less than $B$, or we order the maximum allowed quantity $B$. We next illustrate these two possible cases via Example \ref{sec:numerical_example}.

{\bf Case 1:} The first case ($B=65$) is shown in Fig. \ref{fig:capacitated_policy_1}.
In this case there are $m=2$ local minima up to (and including) the global minimizer $S_m$. Let $y$ denote the initial inventory and apply Eq. \eqref{eq:optimal_policy_cap}. Since $s_2+B\geq S_2$, if $s_1<y<s_2$ we order $x=\min\{S_2-y, B\}$; if $y<s_1$ we order $x=\min\{S_1-y, B\}$. Finally, if $y\geq s_2$, we do not order.

{\bf Case 2:} The second case ($B=71$) is shown in Fig. \ref{fig:capacitated_policy_2}.
\begin{figure}
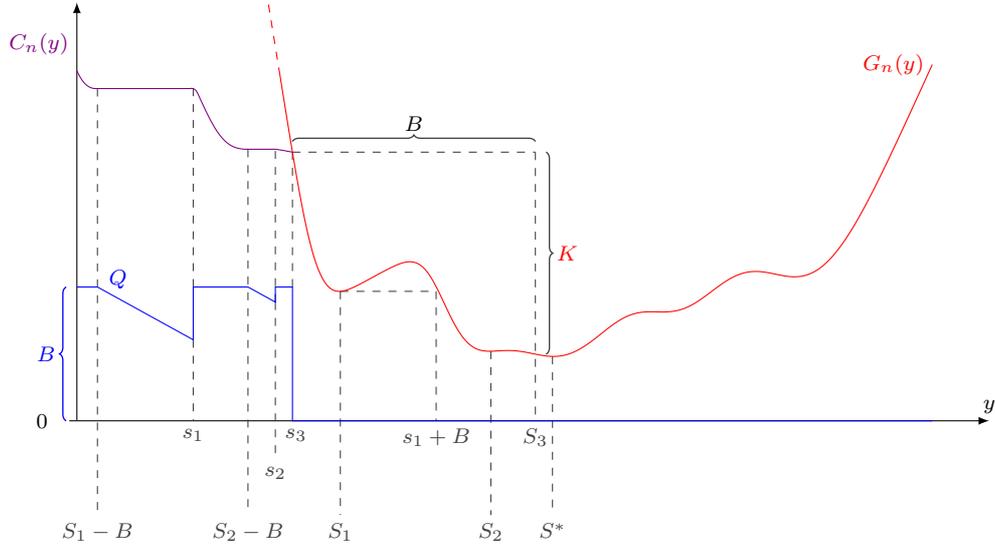

\centering
\include{numerical_example_5}
\caption{Optimal ordering policy in period 1 when $B=71$; note that $G_n(y)$ and $Q$ are not plotted according to the same vertical scale. If $y>s_3$, it is not convenient to order.}
\label{fig:capacitated_policy_2}
\end{figure}
In this case there are $m=3$ local minima up to (and including) the global minimizer $S^*$. Let $y$ denote the initial inventory and apply Eq. \eqref{eq:optimal_policy_cap}. Since capacity $B$ is insufficient to reach the global minimizer $S^*$, if $s_2<y<s_3$ we order $x=S_3-s_3=B$; if $s_1<y<s_2$ we order $x=\min\{S_2-y, B\}$; and if $y<s_1$, we order $x=\min\{S_1-y, B\}$. Finally, if $y\geq s_3$, we do not order.

These two cases exhaust all possible scenarios one may observe when inventory hits level $s_m$.

\section{Numerical example illustrating Lemma \ref{lemma:local_minima}}\label{sec:appendix_5}
\begin{example}\label{example:local_minimum}
Consider a planning horizon of $n=12$ periods; a demand $d_t$ distributed in each period $t=1,\ldots,n$ according to a Poisson law with rate $\lambda_t\in\{151, 152, 58, 78, 134, 13, 22, 161, 43, 55, 110, 37\}$; $K=494$, $v=0$, $h=1$, $p=15$, and $B=128$. 
\end{example}
We focus on period 6, and in Fig. \ref{fig:local_minima_domination} we plot $G_6(y)$ for an initial inventory $y\in(60,145)$. It is clear that at any point $x_0$ in which it is optimal to place an order, if we have sufficient capacity to order beyond $b_1$, we should do so; however, if we do not have sufficient capacity, then we would never order up to $S$, as this point is clearly dominated by $\widehat{S}$. Observe that while $\widehat{S}$ belongs to the QCE of $G_6$ --- illustrated as a dashed line where it departs from $G_6$ --- $S$ does not.
\begin{figure}
\centering
\begin{tikzpicture}[x=0.0380952380952381cm, y=0.035cm]
\footnotesize
\draw [-latex] ([xshift=-0mm] 60.0,0) -- ([xshift=3mm] 210.0,0) node[above] {$y$};
\draw [-latex] ([yshift=-0mm] 60,0.0) -- ([yshift=3mm] 60, 100.0);

\draw [color=black, mark= , style=dotted] (75,1766.5536359408084-1700) -- (75,-5) node[below] {$\widehat{S}$};
\draw [color=black, mark= , style=dotted] (132,1766.1984416651057-1700) -- (132,-5) node[below] {$b_1$};

\draw [color=black, mark= , style=dotted] (101,1767.8614598966587-1700) -- (101,1767.8614598966587-1700+10) node[above] {$S$};

\draw [smooth, color=red, mark= , style=solid] plot coordinates {%
(60, 1795.5301261731734-1700)
(61, 1790.7091494746644-1700)
(62, 1786.4394635524616-1700)
(63, 1782.6912271889444-1700)
(64, 1779.4332865456925-1700)
(65, 1776.6334521779772-1700)
(66, 1774.2587782419914-1700)
(67, 1772.2758402485874-1700)
(68, 1770.6510081093616-1700)
(69, 1769.3507116732396-1700)
(70, 1768.3416964300372-1700)
(71, 1767.5912675290576-1700)
(72, 1767.0675206959327-1700)
(73, 1766.7395590051244-1700)
(74, 1766.5776947604993-1700)
(75, 1766.5536359408084-1700)
(76, 1766.640656777305-1700)
(77, 1766.8137520509806-1700)
(78, 1767.0497746373035-1700)
(79, 1767.3275557045722-1700)
(80, 1767.628006827233-1700)
(81, 1767.9342030685352-1700)
(82, 1768.2314459533977-1700)
(83, 1768.5073051195704-1700)
(84, 1768.7516373743897-1700)
(85, 1768.9565819053414-1700)
(86, 1769.1165305090897-1700)
(87, 1769.2280719228554-1700)
(88, 1769.2899096638741-1700)
(89, 1769.302753199625-1700)
(90, 1769.2691827689778-1700)
(91, 1769.193488731765-1700)
(92, 1769.081486916854-1700)
(93, 1768.940312040593-1700)
(94, 1768.778191853254-1700)
(95, 1768.604205218197-1700)
(96, 1768.4280278158349-1700)
(97, 1768.2596695705329-1700)
(98, 1768.1092081996949-1700)
(99, 1767.9865234552424-1700)
(100, 1767.9010366472062-1700)
(101, 1767.8614598966587-1700)
(102, 1767.875559268303-1700)
(103, 1767.949935508349-1700)
(104, 1768.089825602586-1700)
(105, 1768.2989278191128-1700)
(106, 1768.5792523503414-1700)
(107, 1768.9309991463535-1700)
(108, 1769.3524640476116-1700)
(109, 1769.8399738795104-1700)
(110, 1770.3878507591178-1700)
(111, 1770.988405482439-1700)
(112, 1771.6319595097934-1700)
(113, 1772.306894753329-1700)
(114, 1772.9997301018784-1700)
(115, 1773.6952234002026-1700)
(116, 1774.3764974331916-1700)
(117, 1775.0251883460774-1700)
(118, 1775.621614848895-1700)
(119, 1776.1449664934219-1700)
(120, 1776.5735092582836-1700)
(121, 1776.884806618333-1700)
(122, 1777.0559541959292-1700)
(123, 1777.0638259861446-1700)
(124, 1776.8853300113317-1700)
(125, 1776.49767109361-1700)
(126, 1775.8786182592123-1700)
(127, 1775.0067739943584-1700)
(128, 1773.8618425122293-1700)
(129, 1772.4248938555847-1700)
(130, 1770.6786205202093-1700)
(131, 1768.6075831384671-1700)
(132, 1766.1984416651057-1700)
(133, 1763.4401684697036-1700)
(134, 1760.324239771192-1700)
(135, 1756.844801960646-1700)
(136, 1752.998809561556-1700)
(137, 1748.7861318844207-1700)
(138, 1744.2096258526717-1700)
(139, 1739.275173007498-1700)
(140, 1733.991679325312-1700)
(141, 1728.3710371757654-1700)
(142, 1722.4280494761545-1700)
(143, 1716.1803168250485-1700)
(144, 1709.6480890967403-1700)
(145, 1702.8540836308616-1700)
} 
node[above right] {$G_6(y)$};
 
 \draw [smooth, color=red, mark= , style=dashed] plot coordinates {%
(75, 1766.5536359408084-1700)
(76, 1766.5536359408084-1700)
(77, 1766.5536359408084-1700)
(78, 1766.5536359408084-1700)
(79, 1766.5536359408084-1700)
(80, 1766.5536359408084-1700)
(81, 1766.5536359408084-1700)
(82, 1766.5536359408084-1700)
(83, 1766.5536359408084-1700)
(84, 1766.5536359408084-1700)
(85, 1766.5536359408084-1700)
(86, 1766.5536359408084-1700)
(87, 1766.5536359408084-1700)
(88, 1766.5536359408084-1700)
(89, 1766.5536359408084-1700)
(90, 1766.5536359408084-1700)
(91, 1766.5536359408084-1700)
(92, 1766.5536359408084-1700)
(93, 1766.5536359408084-1700)
(94, 1766.5536359408084-1700)
(95, 1766.5536359408084-1700)
(96, 1766.5536359408084-1700)
(97, 1766.5536359408084-1700)
(98, 1766.5536359408084-1700)
(99, 1766.5536359408084-1700)
(100, 1766.5536359408084-1700)
(101, 1766.5536359408084-1700)
(102, 1766.5536359408084-1700)
(103, 1766.5536359408084-1700)
(104, 1766.5536359408084-1700)
(105, 1766.5536359408084-1700)
(106, 1766.5536359408084-1700)
(107, 1766.5536359408084-1700)
(108, 1766.5536359408084-1700)
(109, 1766.5536359408084-1700)
(110, 1766.5536359408084-1700)
(111, 1766.5536359408084-1700)
(112, 1766.5536359408084-1700)
(113, 1766.5536359408084-1700)
(114, 1766.5536359408084-1700)
(115, 1766.5536359408084-1700)
(116, 1766.5536359408084-1700)
(117, 1766.5536359408084-1700)
(118, 1766.5536359408084-1700)
(119, 1766.5536359408084-1700)
(120, 1766.5536359408084-1700)
(121, 1766.5536359408084-1700)
(122, 1766.5536359408084-1700)
(123, 1766.5536359408084-1700)
(124, 1766.5536359408084-1700)
(125, 1766.5536359408084-1700)
(126, 1766.5536359408084-1700)
(127, 1766.5536359408084-1700)
(128, 1766.5536359408084-1700)
(129, 1766.5536359408084-1700)
(130, 1766.5536359408084-1700)
(131, 1766.5536359408084-1700)} ;
  
\end{tikzpicture}
\caption{Example \ref{example:local_minimum}, plot of function $G_6(y)$ for an initial inventory $y\in(60,145)$; the QCE of $G_6$, when it departs from $G_6$, is illustrated as a dashed line. Observe that $\widehat{S}$ belongs to the QCE of $G_6$, while $S$ does not.}
\label{fig:local_minima_domination}
\end{figure}
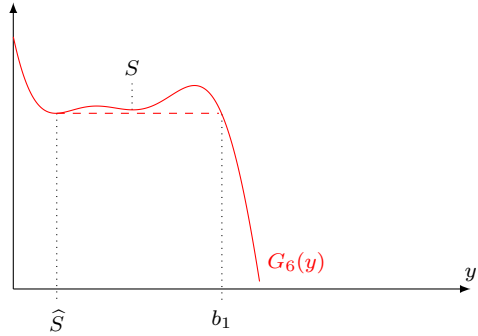 

\section{Example from \cite{Chen1996X}}\label{sec:appendix_4}

We hereby illustrate that an $(s_k,S_k)$ ordering policy is optimal for the numerical example originally presented in \citep[][p. 1015]{Chen1996X} and also investigated in \citep{Chen2004The} under an infinite horizon.

\begin{example}
Consider a planning horizon of $n=20$ periods and a stationary demand $d$ distributed in each period according to the following probability mass function: $\Pr\{d=6\}=0.95$ and $\Pr\{d=7\}=0.05$. Other problem parameters are $K=22$, $B=9$, $h=1$ and $p=10$ and $v=1$; note that, if the planning horizon is sufficiently long, $v$ can be safely ignored. The discount factor is $\alpha=0.9$. 
\end{example}

In Table \ref{tab:shaoxiang_opt} we report the tabulated optimal policy as illustrated in \citep[][p. 417]{Chen2004The}.

\begin{table}
\centering
\begin{tabular}{lrrrrrrrrrrr}
\toprule
Starting inventory level	&-3&-2&-1&0&1&2&3&4&5&6&7\\
Optimal order quantity	&9&8&7&9&8&7&9&8&7&0&0\\
\bottomrule
\end{tabular}
\caption{Optimal policy as illustrated in \citep[][p. 417]{Chen2004The}}
\label{tab:shaoxiang_opt}
\end{table}

In Fig. \ref{fig:shaoxiang_policy} we plot $G_n(y)$ for an initial inventory $y\in(-5,50)$ and $n=20$. The optimal $(s_k,S_k)$ policy is shown in Table \ref{tab:optimal_sk_Sk_policy_shaoxiang}; this is equivalent to the policy illustrated in \citep[][p. 1015]{Chen1996X} and to the stationary policy tabulated in \citep[][p. 417]{Chen2004The}.

\begin{table}
\centering
\begin{tabular}{lc}
\toprule
$s_k$	&$S_k$\\	
\cmidrule(l{2pt}r{2pt}){1-2}
-1		&6\\
2		&9\\
5		&12\\
\bottomrule
\end{tabular}
\caption{Optimal $(s_k,S_k)$ policy for a generic period $t$ of the example in \citep{Chen1996X}.}
\label{tab:optimal_sk_Sk_policy_shaoxiang}
\end{table}

\begin{figure}
\centering
\begin{tikzpicture}[x=0.2cm, y=0.035cm]
\footnotesize
\draw [-latex] ([xshift=-0mm] -5.0,0) -- ([xshift=3mm] 50.0,0) node[above] {$y$};
\draw [-latex] ([yshift=-0mm] -5,0.0) -- ([yshift=3mm] -5, 100.0);
\draw (-5,0.0) -- +(0mm,0mm) -- +(-1mm,0mm) node[left=5pt] {0};

\draw [color=darkgray, mark= , style=dashed] (-3,98.37442975280123) -- (-3,-15) node[below] {$S_1-B$};
\draw [color=darkgray, mark= , style=dashed] (-1,98.37442975280123) -- (-1,0) node[below] {$s_1$};
\draw [color=darkgray, mark= , style=dashed] (6,40.956885426525304) -- (6,-30) node[below] {$S_1$};

\draw [color=darkgray, mark= , style=dashed] (0,98.37442975280123) -- (0,-30) node[below] {$S_2-B$};
\draw [color=darkgray, mark= , style=dashed] (2,98.37442975280123) -- (2,0) node[below] {$s_2$};
\draw [color=darkgray, mark= , style=dashed] (9,145.26820848099007-108.59219886666103) -- (9,-10) node[below] {$S_2$};

\draw[color=darkgray,decoration={brace,raise=4pt},decorate,line width=0.5pt]
  (2,58) -- (11,58) node[midway,black,above=5pt] {$B$};
\draw [color=darkgray, mark= , style=densely dotted] (9,145.26820848099007-108.59219886666103) -- (11,145.26820848099007-108.59219886666103);  
\draw [color=darkgray, mark= , style=dashed] (11,58) -- (11,145.26820848099007-108.59219886666103);  

\draw[color=darkgray,decoration={brace,raise=4pt},decorate,line width=0.5pt]
  (3,70) -- (12,70) node[midway,black,above=5pt] {$B$};  
  
\draw[color=blue,decoration={brace,raise=4pt},decorate,line width=0.5pt]
  (-5,0) -- (-5,27) node[midway,blue,left=5pt] {$B$}; 

\draw [color=darkgray, mark= , style=dashed] (3,74.052255726385) -- (3,-15) node[below] {$S_3-B$};
\draw [color=darkgray, mark= , style=dashed] (5,53.32214130394205) -- (5,0) node[below] {$s_3$};
\draw [color=darkgray, mark= , style=dashed] (12,70) -- (12,-10) node[below] {$S_3$};

 \draw [color=violet, mark= , style=solid] plot coordinates 
 {%
(-5, 193.91434017060305-108.59219886666103)
(-4, 183.91434017060305-108.59219886666103)
(-3, 170.68466703477696-108.59219886666103)
(-2, 170.68466703477696-108.59219886666103)
(-1, 170.68466703477696-108.59219886666103)
(0, 167.26820848099007-108.59219886666103)
(1, 167.26820848099007-108.59219886666103)
(2, 167.26820848099007-108.59219886666103)
(3, 160.08454286044568-108.59219886666103)
(4, 160.08454286044568-108.59219886666103)
(5, 160.08454286044568-108.59219886666103)
 };
 
 \draw[color=violet] (-5, 135-50) node[left] {$C_n(y)$};

\draw [smooth, color=red, mark= , style=dashed] plot coordinates {%
(1, 223.46662297386672-108.59219886666103)
(2, 204.44077131748384-108.59219886666103)
(3, 182.51816172794713-108.59219886666103)
(4, 171.91434017060308-108.59219886666103)
};

\draw [smooth, color=red, mark= , style=solid] plot coordinates {%
(4, 171.91434017060308-108.59219886666103)
(5, 161.91434017060308-108.59219886666103)
(6, 148.68466703477696-108.59219886666103)
(7, 148.96468423815455-108.59219886666103)
(8, 149.96468423815455-108.59219886666103)
(9, 145.26820848099007-108.59219886666103)
(10, 145.96839396745509-108.59219886666103)
(11, 146.96839396745509-108.59219886666103)
(12, 138.08454286044568-108.59219886666103)
(13, 138.79526861325903-108.59219886666103)
(14, 140.66242274181988-108.59219886666103)
(15, 137.39182116333782-108.59219886666103)
(16, 138.74754731297128-108.59219886666103)
(17, 140.633224316909-108.59219886666103)
(18, 134.54388548492778-108.59219886666103)
(19, 135.82543222870783-108.59219886666103)
(20, 138.45775028344354-108.59219886666103)
(21, 136.5591523411612-108.59219886666103)
(22, 138.54170138827772-108.59219886666103)
(23, 141.21341449602068-108.59219886666103)
(24, 136.80154986863565-108.59219886666103)
(25, 138.5627223337551-108.59219886666103)
(26, 141.8662916120153-108.59219886666103)
(27, 141.852605124947-108.59219886666103)
(28, 144.56568232293165-108.59219886666103)
(29, 147.9473778012813-108.59219886666103)
(30, 145.08356294160836-108.59219886666103)
(31, 147.3335974326813-108.59219886666103)
(32, 151.23122953664023-108.59219886666103)
(33, 152.10128242001522-108.59219886666103)
(34, 155.35018774108292-108.59219886666103)
(35, 159.35624145422136-108.59219886666103)
(36, 158.56482242249913-108.59219886666103)
(37, 161.52000545032362-108.59219886666103)
(38, 165.97491458987665-108.59219886666103)
(39, 167.62961705894128-108.59219886666103)
(40, 171.36261768811954-108.59219886666103)
(41, 175.92293968440282-108.59219886666103)
(42, 176.17320398132165-108.59219886666103)
(43, 179.57200374544956-108.59219886666103)
(44, 184.4994984808252-108.59219886666103)
(45, 187.64461177440782-108.59219886666103)
(46, 192.10640173877488-108.59219886666103)
(47, 197.20435659489442-108.59219886666103)
(48, 198.33222408498267-108.59219886666103)
(49, 202.12429562604967-108.59219886666103)
} node[right] {$G_n(y)$};
 
\draw [color=blue, mark= , style=solid] plot coordinates {%
(-5,9*3)
(-4,9*3)
(-3,9*3)
(-2,8*3)
(-1,7*3)
(-1,9*3)
(0,9*3)
(1,8*3)
(2,7*3)
(2,9*3)
(3,9*3)
(4,8*3)
(5,7*3)
(5,0)
(6,0)
(7,0)
(8,0)
(9,0)
(10,0)
(11,0)
(12,0)
(13,0)
(14,0)
(15,0)
(16,0)
(17,0)
(18,0)
(19,0)
(20,0)
(21,0)
(22,0)
(23,0)
(24,0)
(25,0)
(26,0)
(27,0)
(28,0)
(29,0)
(30,0)
(31,0)
(32,0)
(33,0)
(34,0)
(35,0)
(36,0)
(37,0)
(38,0)
(39,0)
(40,0)
(41,0)
(42,0)
(43,0)
(44,0)
(45,0)
(46,0)
(47,0)
(48,0)
(49,0)} 
 node[above] at (-4.0,13.0) {$Q$};
 
\end{tikzpicture}
\caption{Optimal ordering policy for the stationary example in \citep{Chen1996X}; note that $G_n(y)$ and $Q$ are not plotted according to the same vertical scale.}
\label{fig:shaoxiang_policy}
\end{figure}
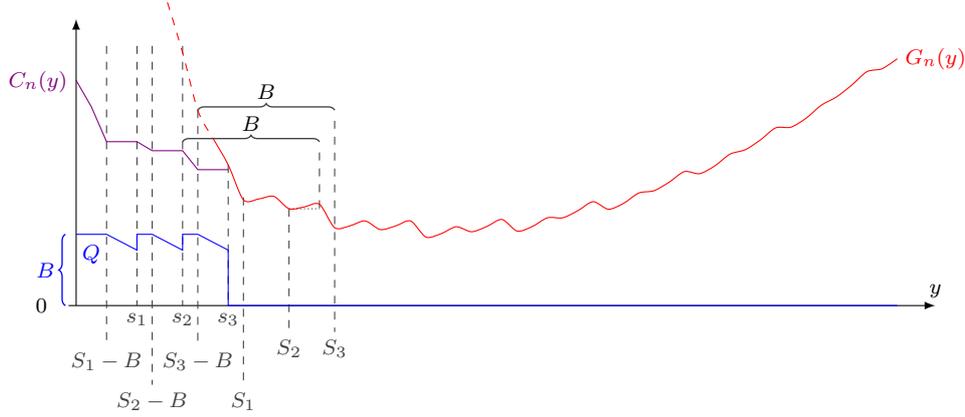

\section{Generating counterexamples to the continuous order property}\label{sec:appendix_6}

Generating counterexamples to the continuous order property is far from trivial. We believe this is the reason why the continuous order property originally conjectured by \cite{Gallego2000Capacitated} has not been so far confirmed or disproved. In this section, we outline the reasoning we followed to generate our counterexample. Our analysis was inspired by the work of \cite{Gallego2004}.

\begin{lemma}\label{lemma:karush1959}
Let $f$ be convex, and $S$ be a minimizer of $f$, then 
\[g(x)\triangleq \min_{y\in[x,x+B]} f(y)-f(x)=\left\{
\begin{array}{lr}
0		&\quad S\leq x\\
f(S)-f(x)		&\quad S-B\leq x\leq S\\
f(x+B)-f(x)		&\quad x\leq S-B
\end{array}\right.
\]
is nondecreasing.
\end{lemma}
\begin{proof}
Following \citep{Karush1959}, $g(x)$ is constant for $S\leq x$; it is nondecreasing for $S-B\leq x\leq S$, since $f(S)$ is constant, and $f$ is nonincreasing in this region; finally, it is nondecreasing for $x\leq S-B$, since $f$ is convex and hence $f(x+B)-f(x)$ is nondecreasing for all $x$.
\end{proof}

Consider $G_n$ and $C_n$ as defined in Eq \eqref{eq:Cn} and Eq. \eqref{eq:Gn}, respectively, and let these functions be $(K,B)$-convex. To show that the continuous order property holds, one must show that $\{x|C_n(x)-(G_n(x)-vx)<0\}$ is the convex set $(-\infty,s_m)$. 

Recall that
\[
\begin{array}{l}
C_n(x)=\min\left\{
\begin{array}{l}
L_n(x)+\int_0^\infty C_{n-1}(x-\xi)f_n(\xi) \mathrm{d}\xi,\\
\min_{x< y \leq x+B} \{K+v(y-x)+L_n(y)+\int_0^\infty C_{n-1}(y-\xi)f_n(\xi) \mathrm{d}\xi\}
\end{array}
\right\},\\
G_n(x)=vx + L_n(x)+\int_0^\infty C_{n-1}(x-\xi)f_n(\xi) \mathrm{d}\xi,\\
C_n(x)=-vx + \min\{G_n(x),K + \min_{x\leq y \leq x+B} G_n(y)\}.
\end{array}
\]
To prove that $\{x|C_n(x)-(G_n(x)-vx)<0\}$ is a convex set, it is sufficient to show that the function 
\[V_n(x)\triangleq C_n(x)-(G_n(x)-vx)\] 
is nondecreasing in $x$ for each $n$. Let $[x]^-\triangleq \min\{0,x\}$, and note that 
\[V_n(x)=[K + \min_{x\leq y \leq x+B} G_n(y) - G_n(x)]^-.\]
One may want to try and show by induction that $V_n(x)$ is nondecreasing in $x$ for each $n$. Let $C_{0}\triangleq 0$, then
\[
\begin{array}{lll}
V_1(x)	& =	[K+\min_{x\leq y \leq x+B}  \{v(y-x) + L_1(y)\} - L_1(x)]^-;
\end{array}
\]
since the unit cost $v$ is linear, and $L_1$ is convex, from Lemma \ref{lemma:karush1959} it follows that $V_1(x)$ is nondecreasing. 
Given this base case, we may then assume that $V_n(x)$ is nondecreasing in $x$, and try to show that $V_{n+1}(x)$ is nondecreasing in $x$. 

First, observe that
\[
\begin{array}{lll}
V_{n+1}(x)	& = [K+\min_{x\leq y \leq x+B} 	& (vy + L_{n+1}(y) + \int_0^\infty C_{n}(y-\xi)f_{n+1}(\xi) \mathrm{d}\xi) \\
			& 						&-(vx + L_{n+1}(x) + \int_0^\infty C_{n}(x-\xi)f_{n+1}(\xi) \mathrm{d}\xi)]^-.
\end{array}
\]
To try and prove $V_{n+1}(x)$ is nondecreasing, we shall first analyse
\[
\begin{array}{ll}
& K +\min_{x\leq y \leq x+B} v(y-x) + C_{n}(y)-C_{n}(x)  \\
& = \min_{x\leq y \leq x+B} \{K + G_n(y) - G_n(x) - V_n(x) + V_n(y)\},
\end{array}
\]
since $C_n(x)=V_n(x)+G_n(x)-vx$. Consider $s_m$ as defined in Lemma \ref{lemma:s_m}, and recall this value denotes an inventory level beyond which no ordering is optimal.
There are three intervals we need to analyse: $x\leq s_m-B$, $s_m-B< x\leq s_m$, and $x> s_m$. Observe that, from the definition of $s_m$ in Lemma \ref{lemma:s_m}, if $x=s_m$, then $K + \min_{x\leq y \leq x+B} G_n(y) - G_n(x)\leq 0$; moreover, by induction hypothesis $V_n(x)$ is assumed nondecreasing, hence $V_n(x)=K + \min_{x\leq y \leq x+B} G_n(y) - G_n(x)$ for $x\leq s_m$.

Let $x\leq s_m-B$; in this interval $V_n(x)=K + \min_{x\leq y \leq x+B} G_n(y) - G_n(x)$, thus
\begin{align}
& {\textstyle\min_{x\leq y \leq x+B} \{K + G_n(y) - G_n(x) - V_n(x) + V_n(y)\}}\nonumber\\[-5pt]
& {\textstyle= \min_{x\leq y \leq x+B} \{K - (\min_{x\leq z \leq x+B} G_n(z)) + (\min_{y\leq w \leq y+B} G_n(w))\}}\nonumber\\[-5pt]	%
& {\textstyle= \min_{x\leq y \leq x+B} \{K - G_n(x+B) + \min_{y\leq w \leq y+B} G_n(w)\}}\nonumber\\[-5pt]						
& {\textstyle= K + \min_{x\leq y \leq x+2B} G_n(y) - G_n(x+B),}\nonumber\\[-5pt]										
& {\textstyle= K + \min_{x+B\leq y \leq x+2B} G_n(y) - G_n(x+B),}\label{region_1}										
\end{align}
because $G_n(x)$ is assumed $(K,B)$-convex and, by Lemma \ref{lemma:decreasing}, it is nonincreasing for $x\leq s_m$, therefore it is also nonincreasing in $(x, x+B)$, since $x\leq s_m-B$.

Let $s_m-B< x\leq s_m$, in this interval $V_n(x)=K + \min_{x\leq z \leq x+B} G_n(z) - G_n(x)$, thus
\begin{align}
&{\textstyle\min_{x\leq y \leq x+B} \{K + G_n(y) - G_n(x) - V_n(x) + V_n(y)\}}\nonumber\\[-5pt]
&{\textstyle = \min_{x\leq y \leq x+B} \{G_n(y) + V_n(y)\} - \min_{x\leq z \leq x+B} G_n(z)}\nonumber\\[-5pt]					%
&{\textstyle = \min_{x\leq y \leq x+B} \{C_n(y) + vy\} - \min_{x\leq z \leq x+B} G_n(z)}\nonumber\\[-5pt]						
&{\textstyle = \min_{s_m< y \leq x+B} \{C_n(y) + vy\} - \min_{s_m< z \leq x+B} G_n(z)=0,}\label{region_2}						
\end{align}
because $G_n(x)$ and $C_n(x)$ are assumed $(K,B)$-convex and, by Lemma \ref{lemma:decreasing}, they are nonincreasing for $x\leq s_m$; and since no ordering is optimal beyond $s_m$, then $\min_{s_m< y \leq x+B} C_n(y)+vy=\min_{s_m< z \leq x+B} G_n(z)$. 

Let $x> s_m$, in this interval $K + \min_{x\leq y \leq x+B} G_n(y) - G_n(x)> 0$, hence $V_n(x)=0$, $V_n(y)=0$, and
\begin{align}
&{\textstyle \min_{x\leq y \leq x+B} \{K + G_n(y) - G_n(x) - V_n(x) + V_n(y)\}}\nonumber\\[-5pt]
&{\textstyle = K + \min_{x\leq y \leq x+B} G_n(y) - G_n(x)> 0.}\label{region_3}
\end{align}

Equipped with Eq. \eqref{region_1}, \eqref{region_2}, and \eqref{region_3} for the intervals we considered, it is immediate to see that 
\[\left[ K +\min_{x\leq y \leq x+B} v(y-x) + C_{n}(y)-C_{n}(x) \right]^-=\left\{
\begin{array}{lr}
V_n(x+B)							&\quad x\leq s_m-B\\
0								&\quad s_m-B< x\leq s_m\\
0								&\quad x> s_m
\end{array}\right.
\] 
is nondecreasing. However, it is not possible to determine if $[K+\min_{x\leq y \leq x+B} v(y-x)+\int_0^\infty (C_{n}(y-\xi)-C_{n}(x-\xi))f_{n+1}(\xi) \mathrm{d}\xi]^-$ is nondecreasing; and reintroducing term $\min_{x\leq y \leq x+B} L_{n+1}(y) - L_{n+1}(x)$ only worsens the matter. But because of the behavior of $\left[ K +\min_{x\leq y \leq x+B} v(y-x) + C_{n}(y)-C_{n}(x) \right]^-$ in intervals $s_m-B< x\leq s_m$ and $x\leq s_m-B$, one may observe that a $V_{n+1}(x)$ function featuring some decreasing regions may be produced by the convolution $\int_0^\infty (C_{n}(y-\xi)-C_{n}(x-\xi))f_{n+1}(\xi) \mathrm{d}\xi$, provided demand is sufficiently ``lumpy.'' In other words, the instance must feature demand whose probability mass function features some values larger than $B$ possessing non negligible probability mass. A demand that is so structured may ensure that the convolution ``bends'' sufficiently $V_{n+1}(x)$ beyond $s_m$ so that it turns negative.

On the basis of this observation, we have generated several random instances as follows. The fixed ordering cost is a randomly generated value uniformly distributed between 1 and 500; holding cost is 1; penalty cost is a randomly generated value uniformly distributed between 1 and 30; the ordering capacity is a randomly generated value uniformly distributed between 20 and 200; demand distribution in each period is obtained as follows: the probability mass function comprises only four values in the support, one of these values must fall below the given order capacity, the other three values must fall above, and be smaller or equal to 300; probability masses are then allocated uniformly to each of these values. 
The Java code to generate instances that violate the continuous order property is available on \url{http://gwr3n.github.io/jsdp/}.\footnote{File \url{https://github.com/gwr3n/jsdp/blob/master/jsdp/src/main/java/jsdp/app/standalone/stochastic/capacitated/CapacitatedStochasticLotSizingFast.java}} 

\section{Expected demand values in our test bed}\label{sec:appendix_3}

Expected demand values for demand patterns in our test bed are shown in Table \ref{table:demand_data}.

\begin{table}
\resizebox{\textwidth}{!}{
\scriptsize
\begin{filecontents*}{demandData.csv}
STA,30 ,30 ,30 ,30 ,30 ,30 ,30 ,30 ,30 ,30 ,30 ,30 ,30 ,30 ,30 ,30 ,30 ,30 ,30 ,30

LC1,46 ,49 ,50 ,50 ,49 ,46 ,42 ,38 ,35 ,33 ,30 ,28 ,26 ,23 ,21 ,18 ,14 ,11 ,8 ,6

LC2,7 ,9 ,11 ,13 ,17 ,22 ,24 ,26 ,32 ,34 ,36 ,41 ,44 ,47 ,48 ,50 ,50 ,49 ,47 ,44

SIN1,47 ,30 ,13 ,6 ,13 ,30 ,47 ,54 ,47 ,30 ,13 ,6 ,13 ,30 ,47 ,30 ,15 ,8 ,11 ,30

SIN2,36 ,30 ,24 ,21 ,24 ,30 ,36 ,39 ,36 ,30 ,24 ,21 ,24 ,30 ,36 ,31 ,24 ,21 ,26 ,33

RAND,63 ,27 ,10 ,24 ,1 ,23 ,33 ,35 ,67 ,7 ,14 ,41 ,4 ,63 ,26 ,45 ,53 ,25 ,10 ,50

EMP1,5 ,15 ,46 ,140 ,80 ,147 ,134 ,74 ,84 ,109 ,47 ,88 ,66 ,28 ,32 ,89 ,162 ,36 ,32 ,50

EMP2,14 ,24 ,71 ,118 ,49 ,86 ,152 ,117 ,226 ,208 ,78 ,59 ,96 ,33 ,57 ,116 ,18 ,135 ,128 ,180

EMP3,13 ,35 ,79 ,43 ,44 ,59 ,22 ,55 ,61 ,34 ,50 ,95 ,36 ,145 ,160 ,104 ,151 ,86 ,123 ,64

EMP4,15 ,56 ,19 ,84 ,136 ,67 ,67 ,155 ,87 ,164 ,194 ,67 ,65 ,132 ,35 ,131 ,133 ,36 ,173 ,152
\end{filecontents*}
\csvautobooktabular[table head=\toprule Pattern & \multicolumn{19}{c}{Expected demand values}\\\midrule\csvlinetotablerow\\, table foot=\\\midrule Period&1&2&3&4&5&6&7&8&9&10&11&12&13&14&15&16&17&18&19&20\\\bottomrule]{demandData.csv}
}
\caption{Expected demand values for demand patterns in our test bed.}
\label{table:demand_data}
\end{table}

\end{appendices}

\bibliographystyle{plainnat}
\bibliography{bibliography}
\end{document}